\documentclass[a4paper]{amsart}
\usepackage[utf8]{inputenc}
\usepackage{amssymb,mathtools,hyperref,cases,enumerate}
\usepackage[shortlabels]{enumitem}

\usepackage{xcolor}

\usepackage[
	backend=biber,
	style=alphabetic,
]{biblatex}
\AtEveryBibitem{\clearlist{language}}
\numberwithin{equation}{section}

\newtheorem{thm}{Theorem}[section]
\newtheorem{prop}[thm]{Proposition}
\newtheorem{lem}[thm]{Lemma}
\theoremstyle{definition}
\newtheorem{defn}[thm]{Definition}
\newtheorem{rmk}[thm]{Remark}

\DeclareMathOperator{\Id}{Id}

\DeclareMathOperator{\cupl}{cupl}

\newcommand{\real}{\mathbb{R}}
\newcommand{\nat}{\mathbb{N}}

\newcommand{\Tan}{\mathrm{T}}
\newcommand{\Dif}{\mathrm{D}}
\newcommand{\dif}{\mathrm{d}}
\newcommand{\End}{\mathcal{L}}

\DeclarePairedDelimiter{\abs}{\lvert}{\rvert}
\DeclarePairedDelimiter{\norm}{\lVert}{\rVert}
\DeclarePairedDelimiter{\parens}{(}{)}
\DeclarePairedDelimiter{\set}{\{}{\}}
\DeclarePairedDelimiter{\brackets}{\lbrack}{\rbrack}
\DeclarePairedDelimiter{\cbrackets}{\{}{\}}
\DeclarePairedDelimiter{\angles}{\langle}{\rangle}
\DeclarePairedDelimiter{\cci}{\lbrack}{\rbrack}
\DeclarePairedDelimiter{\coi}{\lbrack}{\lbrack}

\DeclarePairedDelimiter{\ooi}{\rbrack}{\lbrack}

\title[Concentrated solutions to $(\mathrm{SBP}_\epsilon)$]{Concentrated solutions to the Schrödinger--Bopp--Podolsky system with a positive potential}

\begin{document}

\author{Gustavo de Paula Ramos}
\address[G. de Paula Ramos]{Instituto de Matemática e Estatística, Universidade de São Paulo, Rua do Matão, 1010, 05508-090 São Paulo SP, Brazil}
\email{gpramos@ime.usp.br, gustavopramos@gmail.com}
\urladdr{http://gpramos.com}

\begin{abstract}
Consider the Schrödinger--Bopp--Podolsky system
\begin{equation}\tag{$\mathrm{SBP}_\epsilon$}\label{SBP_epsilon}
\begin{cases}
	-\epsilon^2\Delta u+\parens{V+K\phi}u=u\abs{u}^{p-1};\\
	\Delta^2\phi-\Delta\phi=4\pi K u^2
\end{cases}
~\text{in}~\real^3
\end{equation}
for sufficiently small $\epsilon>0$, where $V,K\colon\real^3\to\coi{0,\infty}$; $p\in\ooi{1,5}$ are fixed and we want to solve for $u,\phi\colon\real^3\to\real$. Under certain hypotheses, we estimate the multiplicity of solutions in function of a critical manifold of $V$ and we establish the existence of solutions concentrated around critical points of $V$.


\end{abstract}

\date{\today}
\maketitle
\tableofcontents

\section{Introduction}

The behavior of a free quantum particle with mass $m$ and charge density $K$ subject to a potential $\tilde{V}$ and an electric potential $\phi$ may be modeled by
\begin{equation}\label{Equation:NSE}
	i\hbar\partial_t\psi
	=
	-\frac{\hbar^2}{2m}\Delta_x\psi+\parens{K\phi+\tilde{V}}\psi-\psi|\psi|^{p-1},
\end{equation}
where
\[
	\real^3\times\real
	\ni
	\parens{x,t}
	\mapsto
	\psi(x,t)
	\in
	\mathbb{C}
\]
denotes the wave function of the particle and $\hbar\in\ooi{0,\infty}$ is a fixed constant. By restricting ourselves to search for \emph{standing wave} solutions to \eqref{Equation:NSE} with a known energy $E\in\real$, i.e., wave functions of the form
$
	\psi(x,t)=u(x)e^{iEt/\hbar}
$
for an unknown $u\colon\real^3\to\real$, it becomes sufficient to solve
\[
	-\epsilon^2 \Delta u + \parens{V + K \phi} u - u \abs{u}^{p - 1}
	=
	0
	~\text{in}~
	\real^3,
\]
where $V:=\tilde{V}+E$ and
$\epsilon := \parens{2m}^{-1/2} \hbar$.

Let us suppose that $\epsilon$ may be arbitrarily small and that $\phi$ denotes the electric potential generated by the particle itself. In the trivial case $K\equiv 0$, this scenario is modeled by the \emph{Nonlinear Schrödinger} equation
\begin{equation} \tag{$\mathrm{NLS}_\epsilon$} \label{NLS_epsilon}
	-\epsilon^2 \Delta u + V u = u \abs{u}^{p - 1}
	~\text{in}~
	\real^3.
\end{equation}
\sloppy
If we consider Maxwell's theory of electromagnetism, then we obtain the \emph{Schrödinger--Poisson--Slater} system
\begin{equation}\tag{$\mathrm{SPS}_\epsilon$}\label{SPS_epsilon}
\begin{cases}
	-\epsilon^2\Delta u+\parens{V+K\phi}u=u|u|^{p-1};\\
	-\Delta\phi=4\pi K u^2\\
\end{cases}
~\text{in}~\real^3.
\end{equation}
Finally, we obtain the \emph{Schrödinger--Bopp--Podolsky} system \eqref{SBP_epsilon} if we consider the Bopp--Podolsky theory of electromagnetism instead (for more details, see \cite[Section 2]{dAveniaSiciliano2019}).

The study of the Schrödinger--Bopp--Podolsky system has gained traction in the mathematical literature only recently after the publication of \cite{dAveniaSiciliano2019}, where d'Avenia and Siciliano introduced the problem; presented the appropriate functional framework; established the existence/non-existence of solutions to
\[
\begin{cases}
	-\Delta u+\omega u + q^2 \phi u = u\abs{u}^{p-2};\\
	\Delta^2 \phi - a^2 \Delta \phi = 4 \pi u^2
\end{cases}
~\text{in}~\real^3
\]
according to the values of the parameters $q, p$ for fixed $a, \omega \in \ooi{0, \infty}$ and proved that its solutions tend to solutions of the Schrödinger--Poisson--Slater system as $a \to 0^+$.

In particular, there is a large and increasing number of studies about this system in the presence of a potential $V \colon \real^3 \to \coi{0, \infty}$. For instance, \cite{ChenTang2020, LiPucciTang2020, Zhang2022, HuWuTang2023, MascaroSiciliano2023, FigueiredoSiciliano2023}. Among those, the ones which considered contexts closest to ours are \cite{MascaroSiciliano2023, FigueiredoSiciliano2023}, where the authors established multiplicity results for solutions in the \emph{semiclassical limit}, i.e., when $\epsilon > 0$ is sufficiently small. In \cite{MascaroSiciliano2023}, Mascaro and Siciliano use a technique due to Benci and Cerami to estimate the number of solutions in function of the Lusternik--Schnirelmann category of
\[
	M:=\set{x \in \real^3 : V \parens{x}=\inf V>0}.
\]
Among other results, Figueiredo and Siciliano employed the Lusternik--Schnirelmann and Morse theories in \cite{FigueiredoSiciliano2023} to establish multiplicity results of negative solutions.

On the other hand, we know of few studies which considered the Schrödinger--Bopp--Podolsky system in the presence of a (possibly non-constant) charge density $K \colon \real^3 \to \coi{0, \infty}$, that is, \cite{TengYan2021, JiaLiChen2022, Peng2022}. To the best of our knowledge, \cite{Peng2022} is the
only paper which simultaneously considered a potential $V$ and a charge density $K$ in the context of the Schrödinger--Bopp--Podolsky system. In this paper, Peng established the existence of ground-state solutions/infinite high energy solutions and studied the existence/multiplicity of sign-changing solutions by considering multiple techniques. We remark that, unlike the present text, Peng did not consider semiclassical states or used any perturbative techniques.

In this context, the goal of this paper is to show that concentration results which Ianni and Vaira proved to hold for \eqref{SPS_epsilon} in \cite{IanniVaira2008} also hold for \eqref{SBP_epsilon}. We also estimate the number of solutions to \eqref{SBP_epsilon} analogously as in Ambrosetti, Malchiodi and Secchi's result for \eqref{NLS_epsilon} in \cite{AmbrosettiMalchiodiSecchi2001}.

Let us state our hypotheses on $V, K \colon \real^3 \to \coi{0, \infty}$. We suppose throughout the text that  $V$ satisfies
\begin{enumerate}[label={$\parens{V_\arabic*}$}]
\item \label{V_1}
	$V$ is of class $C^2$ and $\norm{V}_{C^2}<\infty$;
\item \label{V_2}
	$\inf V>0$.
\end{enumerate}
Similarly, we always suppose that
\begin{enumerate}[label={$\parens{K_\arabic*}$}]
\item \label{K_1}
	$K$ is continuous and bounded.
\end{enumerate}

Before stating our results about concentrated solutions, we need to introduce a certain solution to
\begin{equation}\label{Equation:SimpleProblem}
	-\Delta u+u=u\abs{u}^{p-1}~\text{in}~\real^3.
\end{equation}
It is well-known that there exists a unique positive spherically symmetric function that vanishes at infinity and solves \eqref{Equation:SimpleProblem} (see the theorem at \cite[p. 23]{Kwong1989}), which we will always denote by
\[
	U \colon \real^3 \to \ooi{0, \infty}.
\]
Furthermore, it follows from \cite[Proposition 4.1]{GidasNiNirenberg1981} that
\begin{equation}\label{Equation:ExponentialDecay}
	\text{there exists}~
	r\in\ooi{0,\infty}
	~\text{such that}~
	U\parens{x}, \abs{\nabla U\parens{x}}
	\lesssim
	\frac{e^{-\abs{x}}}{\abs{x}}
	~\text{whenever}~
	\abs{x}>r,
\end{equation}
and so $U\in H^1$.

Let us state our first result, which concerns the existence of a family of solutions concentrated around non-degenerate critical points of $V$.

\begin{thm}\label{Theorem:SolutionsNonDegenerate}
If $x_0$ is a non-degenerate critical point of $V$, then there exist $\epsilon_0\in\ooi{0,1}$ and $\set{u_\epsilon}_{\epsilon\in\ooi{0,\epsilon_0}}\subset H^1$ such that given $\epsilon\in\ooi{0,\epsilon_0}$, $\parens{u_\epsilon,\varphi_\epsilon}$ is a weak solution to \eqref{SBP_epsilon} and
\begin{equation} \label{Equation:Concentration}
	\norm*{
		u_\epsilon
		-
		U_{\lambda_{x_0}}\parens*{
			\frac{1}{\epsilon}\parens{\cdot - x_0}
		}
	}_{H^1}
	\to
	0
	~\text{as}~
	\epsilon \to 0,
\end{equation}
where
$\lambda_{x_0} := V \parens{x_0}^{1/2}$;
$U_{\lambda} := \lambda^{2 / \parens{p - 1}} U \parens{\lambda \cdot}$
and
$\varphi_\epsilon:=\phi_{\epsilon, u_\epsilon^2}$
is defined in \eqref{Equation:phi}.
\end{thm}

Due to \eqref{Equation:Concentration}, we say that the family of solutions provided by Theorem \ref{Theorem:SolutionsNonDegenerate} is \emph{concentrated} around $x_0$. We remark that if a family of solutions to \eqref{SBP_epsilon} is concentrated around $x_0 \in \real^3$, then $\nabla V \parens{x_0}=0$ (see Proposition \ref{Proposition:Concentration}).

Our next result establishes a sufficient condition for the existence of a family of solutions concentrated around a degenerate critical point of $V$. Technically, this is done by following Ianni and Vaira's approach of comparing how many derivatives of $V$ and $K$ vanish at the same point.

Let us introduce further hypotheses on $V, K$ before proceeding to our next result. Consider the hypotheses
\begin{enumerate}[label={$\parens{V_\arabic*}$}]
\setcounter{enumi}{2}
\item \label{V_3}
$x_0$ is a critical point of $V$ and either
	\begin{itemize}
	\item
		\begin{itemize}
		\item
			$n\geq 4$ is an even integer;
		\item
			$V$ is of class $C^n$;
		\item
			$\norm{V}_{C^n}<\infty$;
		\item
			$
			\parens{
				\partial^n_1 V \parens{x_0},
				\partial^n_2 V \parens{x_0},
				\partial^n_3 V \parens{x_0}
			}
			\neq
			0
			$
			but the mixed partial derivatives of $V$ with order $n$ vanish at $x_0$ and
		\item
			$\Dif_{x_0}^l V=0$ for $l \in \set{1, \ldots, n - 1}$ or
		\end{itemize}
	\item
		$V$ is smooth; every derivative of $V$ is bounded and every derivative of $V$ vanishes at $x_0$, in which case we set $n=\infty$
	\end{itemize}
\end{enumerate}
and
\begin{enumerate}[label={$\parens{K_\arabic*}$}]
\setcounter{enumi}{2}
\item \label{K_2}
	either
	\begin{itemize}
	\item
		\begin{itemize}
		\item
			$m\geq 2$ is an even integer;
		\item
			$K$ is of class $C^m$;
		\item
			$\norm{K}_{C^m}<\infty$;
		\item
			$
			\parens{
				\partial^m_1 K \parens{x_0},
				\partial^m_2 K \parens{x_0},
				\partial^m_3 K \parens{x_0}
			}
			\neq
			0
			$
			but the mixed partial derivatives of $K$ with order $m$ vanish at $x_0$ and
		\item
			$\Dif_{x_0}^l K = 0$ for $l \in \set{0, \ldots, m - 1}$ or
		\end{itemize}
	\item
		$K$ is smooth; every derivative of $K$ is bounded and every derivative of $K$ vanishes at $x_0$, in which case we set $m=\infty$.
	\end{itemize}
\end{enumerate}

We can finally state our second result.

\begin{thm} \label{Theorem:SolutionsDegenerate}
Suppose that \ref{V_3}, \ref{K_2} hold and $\gamma:=\min\parens{n,2m+3}<\infty$. If $n>2m+3$, then suppose further that $\ker \Dif_{x_0} g = 0$, where $g \colon \real^3 \to \real^3$ is defined as
\[
g\parens{\xi}
=
\frac{1}{\parens{m!}^2}
\sum_{
	\substack{0 \leq \alpha, \beta \leq m; \\ 1 \leq i, j, l \leq 3}
} \brackets*{
	\partial_j^m K\parens{x_0}
	\partial_l^m K\parens{x_0}
	\binom{m}{\alpha}
	\binom{m}{\beta}
	\xi_j^{m-\alpha}
	\xi_l^{m-\beta}
	\tilde{C}_{\alpha,\beta,i,j,l}
}e_i
\]
and
\[
	\tilde{C}_{\alpha,\beta,i,j,l}
	:=
	\int\int
		x_j^\alpha y_l^\beta
		U_{\lambda_{x_0}} \parens{x}^2
		U_{\lambda_{x_0}} \parens{y}
		\partial_i U_{\lambda_{x_0}} \parens{y}
	\dif x\dif y
\]
for every $i, j, l \in \set{1, 2, 3}$;
$\alpha, \beta \in \set{0, \ldots, m}$. In this situation, the conclusion of Theorem \ref{Theorem:SolutionsNonDegenerate} holds once again.
\end{thm}

Analogously as in \cite[Remark 1.2]{IanniVaira2008}, the condition
\[
\parens{
	\partial^m_1 K \parens{x_0},
	\partial^m_2 K \parens{x_0},
	\partial^m_3 K \parens{x_0}
}
=
\parens{1, \delta, \delta}
\]
for $\delta \in \real$ with sufficiently small absolute value is sufficient to have $\ker \Dif_{x_0} g = 0$ in the context of Theorem \ref{Theorem:SolutionsDegenerate}.

Let us recall two notions before stating our multiplicity result. The first is that of (non-degenerate) critical manifolds.
\begin{defn}\label{Definition:CriticalManifold}
Let $X$ be a Banach space and let $I\in C^2\parens{X}$. We say that $Y\subset X$ is a \emph{critical manifold} of $I$ when $Y$ is a submanifold of $X$ and $\nabla I\parens{y}=0$ for every $y\in Y$. If $Y$ is a critical manifold of $I$, then it is said to be \emph{non-degenerate} when given $y\in Y$, $\Tan_yY=\ker\Dif^2_yI$ and
\[
X \ni x \mapsto \Dif^2_y I \parens{x, \cdot} \in X'
\]
is a Fredholm operator with index $0$.
\end{defn}

The second is the definition of the cup-length.

\begin{defn}
Suppose that $M \subset \real^3$. If $\check{H}^* \parens{M} = 0$, then the \emph{cup-length} of $M$ is defined as $\cupl \parens{M} = 1$, where $\check{H}^* \parens{M}$ denotes the Alexander--Spanier cohomology of $M$. Otherwise, we set
\[
	\cupl \parens{M} = \sup \set{
		k \in \nat :
		\exists \alpha_1, \ldots, \alpha_k \in \check{H}^* \parens{M};~
		\alpha_1 \cup \ldots \cup \alpha_k \neq 0
	},
\]
where $\cup$ denotes the cup product.
\end{defn}

In a similar spirit to \cite[Theorem 1.1]{MascaroSiciliano2023}, our last theorem estimates the multiplicity of concentrated solutions around points in a compact non-degenerate critical manifold of $V$ in function of its cup-length.
\begin{thm}\label{Theorem:MultiplicitySolutions}
If $M$ is a compact non-degenerate critical manifold of $V$ and $\epsilon\in\ooi{0,1}$ is sufficiently small, then \eqref{SBP_epsilon} has at least $\cupl\parens{M} + 1$ weak solutions that concentrate around points of $M$.
\end{thm}

Similarly as in \cite{AmbrosettiMalchiodi2006, IanniVaira2008}, our results are obtained through an application of the classical Lyapunov--Schmidt reduction technique (for a quick overview, see Section \ref{LyapunovSchmidt}). To the best of our knowledge, this is the first time that this approach was considered to construct solutions to the Schrödinger--Bopp--Podolsky system. In order to avoid unnecessary repetition of non-crucial arguments, our reasoning along the text relies upon results in the aforementioned references.

We remark that Theorem \ref{Theorem:MultiplicitySolutions} is analogous to the proposed generalization of \cite[Theorem 1.1]{IanniVaira2008} in the last paragraph of \cite[Section 4]{IanniVaira2008} (see also \cite[Theorem 8.5]{AmbrosettiMalchiodi2006}). Theorems \ref{Theorem:SolutionsNonDegenerate}, \ref{Theorem:SolutionsDegenerate} are respectively analogous to \cite[Theorems 4.1, 4.2]{IanniVaira2008}, whose proofs are sketched in \cite[p. 593]{IanniVaira2008}.

Let us introduce a notation before proceeding to more technical considerations. If $f \colon \real^3 \to \ooi{0,\infty}$ satisfies
$0 < \inf f \leq \sup f < \infty$,
then we define the Hilbert space $H^1_f$ as the completion of $C_c^\infty$ with respect to
\[
	\angles{u \mid w}_{H^1_f}
	:=
	\int\parens{\nabla u \cdot \nabla w + f u w}.
\]
Note that it is easy to check that $\norm{\cdot}_{H^1_f}$ is equivalent to $\norm{\cdot}_{H^1}$, so we will canonically identify $H^1_f$ with $H^1$.

We explain in Section \ref{VariationalFramework} that in order to obtain weak solutions to \eqref{SBP_epsilon}, it suffices to search for critical points of the \emph{energy functional} $E_\epsilon \colon H^1 \to \real$ given by
\begin{equation}\label{Equation:EnergyFunctional}
	E_\epsilon\parens{u}
	=
	\frac{1}{2} \norm{u}_{H^1_{V_\epsilon}}^2
	+
	\frac{\epsilon^3}{4} \int \parens{
		K_\epsilon \phi_{\epsilon, u^2} u^2
	}	
	-
	\frac{1}{p+1} \norm{u}_{L^{p+1}}^{p+1},
\end{equation}
where
\begin{equation}\label{Equation:phi}
	\phi_{\epsilon, u w}:=\parens{K_\epsilon u w} \ast \kappa_\epsilon
\end{equation}
for every $u, w \in H^1$; $\ast$ denotes the convolution;
$
	\kappa\parens{x} := \abs{x}^{-1}\brackets{1-\exp\parens{-\abs{x}}}
$
for every $x\in\real^3\setminus\set{0}$;
$V_\epsilon := V \parens{\epsilon \cdot}$ and similarly for $K_\epsilon$, $\kappa_\epsilon$. Similarly, critical points of the functionals $I_\epsilon, J_\epsilon \in C^2\parens{H^1}$ given by
\[
	I_\epsilon\parens{u}
	=
	\frac{1}{2} \norm{u}_{H^1_{V_\epsilon}}^2	
	-
	\frac{1}{p+1} \norm{u}_{L^{p+1}}^{p+1};
\]
\[
	J_\epsilon\parens{u}
	=
	\frac{1}{2} \norm{u}_{H^1_{V_\epsilon}}^2
	+
	\frac{\epsilon^2}{4}
	\int \cbrackets*{
		K_\epsilon
		\brackets*{\parens{K_\epsilon u^2} \ast \abs{\cdot}^{-1}}
		u^2
	}
	-
	\frac{1}{p+1} \norm{u}_{L^{p+1}}^{p+1}
\]
are respectively associated to weak solutions to (\ref{NLS_epsilon}; \ref{SPS_epsilon}).

On one hand, we can interpret both functionals $E_\epsilon$ and $J_\epsilon$ as perturbations of $I_\epsilon$. On the other hand, unlike $\abs{\cdot}^{-1}$, $\kappa$ is bounded and not $(-1)$-homogeneous, so $E_\epsilon$ is obtained from $I_\epsilon$ by summing a perturbation of order $\epsilon^3$ (see Lemma \ref{Lemma:FourTerms}), while $J_\epsilon$ is obtained by summing a perturbation of order $\epsilon^2$. This discrepancy is the reason why the sufficient condition at Theorem \ref{Theorem:SolutionsDegenerate} differs from that in \cite[Theorem 4.2]{IanniVaira2008}. Indeed, it follows from the proof of Lemma \ref{Lemma:ExpansionDegenerate} that we have to compare $n$ with $2m+3$, so the case $n=2m+3$ is automatically ruled out. The analogous comparison in \cite[Theorem 4.2]{IanniVaira2008} is done between $n$ and $2m+2$, so one also has to consider the case $n=2m+2$.

Let us finish with a comment on the organization of the text. The preliminaries are done in Section \ref{Preliminaries}: we explain our notion of weak solutions; we develop a variational characterization for them; we do a quick review of the Lyapunov--Schmidt reduction and we highlight a related problem. In Section \ref{MultiplicityResult}, we develop the details of the Lyapunov--Schmidt reduction to prove Theorem \ref{Theorem:MultiplicitySolutions}. Similarly, Section \ref{ConcentrationResults} is concerned with Theorems \ref{Theorem:SolutionsNonDegenerate} and \ref{Theorem:SolutionsDegenerate}.

\subsection*{Notation}

Given functions $f,g\colon X\to[0,\infty[$, we write $f\parens{x}\lesssim g\parens{x}$ for every $x\in X$
when there exists $C\in\ooi{0,\infty}$ such that $f\parens{x}\leq Cg\parens{x}$ for every $x\in X$. We denote the canonical basis of $\real^3$ by $e_1, e_2, e_3 \in \real^3$. Given $r\in\ooi{0,\infty}$, we define $B_r=\set{x\in\real^3: \abs{x}<r}$. We only consider functional spaces over $\real^3$, so the considered domain is omitted from the notation. Likewise, integrals are implied to be taken in $\real^3$ unless denoted otherwise. The functional spaces $H^1$ and $\mathcal{X}$ are obtained as the respective Hilbert space completions of $C_c^\infty$ with respect to the inner products
\[
	\angles{u \mid v}_{H^1}
	:=
	\int\parens{
		\nabla u\cdot\nabla v
		+
		uv
	}
	\quad
	\text{and}
	\quad
	\angles{u \mid v}_{\mathcal{X}}
	:=
	\int\parens{
		\Delta u\Delta v
		+
		\nabla u\cdot\nabla v
	}.
\]

\subsection*{Acknowledgements}

This study was financed in part by the Coordenação de Aperfeiçoamento de Pessoal de Nível Superior - Brasil (CAPES) - Finance Code 001.


\section{Preliminaries}\label{Preliminaries}
\subsection{Variational framework} \label{VariationalFramework}

Let us begin with some considerations about \eqref{Equation:phi}.
\begin{rmk} \label{Remark:phi}
Consider a fixed $\epsilon \in \ooi{0, 1}$. We have
\begin{itemize}
\item
$\phi_{\epsilon, u w} \in \mathcal{X}$ for every $u, w \in H^1$ and
\item
the application
\[
	\parens{u_1, u_2, u_3, u_4} \in \parens{H^1}^4
	\mapsto
	\int \parens{
		K_\epsilon \phi_{\epsilon, u_1 u_2} u_3 u_4
	}
	=
	\int \parens{
		K_\epsilon \phi_{\epsilon, u_3 u_4} u_1 u_2
	}
\]
is multilinear.
\end{itemize}
\end{rmk}

We proceed to a discussion about our notion of weak solutions. Given $u \in H^1$, we define a \emph{weak solution} to
\begin{equation}\label{Equation:Electrostatic}
	\Delta^2\phi-\Delta\phi=4\pi K u^2~\text{in}~\real^3
\end{equation}
as a $\phi\in\mathcal{X}$ such that
\[
	\int \parens{
		\Delta \phi \Delta w + \nabla \phi \cdot \nabla w
	}
	=
	4 \pi \int \parens{K u^2 w}
\]
for every $w\in C_c^\infty$. Analogously, we say that $\parens{u,\phi}$ is a \emph{weak solution} to \eqref{SBP_epsilon} when $u\in H^1$, $\phi$ is a weak solution to \eqref{Equation:Electrostatic} and
\[
	\int \brackets{
		\epsilon^2 \nabla u \cdot \nabla w + \parens{V + K \phi} u w
	}
	=
	\int \parens{u \abs{u}^{p-1} w}
\]
for every $w\in C_c^\infty$.

It follows from the Riesz Theorem that $\phi_{1,u^2}$ is the unique weak solution to \eqref{Equation:Electrostatic}. Therefore, we naturally associate solutions to
\[
	\begin{cases}
		\int \brackets{
			\epsilon^2 \nabla u \cdot \nabla w
			+
			\parens{V + K \phi_{1,u^2}} u w
		}
		-
		\int \parens{u \abs{u}^{p-1} w}
		=
		0
		~\text{for every}~
		w\in C_c^\infty;
		\\
		u\in H^1
	\end{cases}
\]
to weak solutions to \eqref{SBP_epsilon}. By considering the change of variable $x\mapsto \epsilon x$, we see that the previous problem is equivalent to
\begin{multline}\label{Equation:WeakSolution}
	\begin{cases}
		\angles{u \mid w}_{H^1_{V_\epsilon}}
		+
		\int \parens{
			\epsilon^3 K_\epsilon \phi_{\epsilon,u^2} u w
		}
		-
		\int \parens{u \abs{u}^{p-1} w}
		=
		0
		~\text{for every}~
		w\in C_c^\infty;
		\\
		u \in H^1.
	\end{cases}
\end{multline}
In view of this discussion, we will restrict ourselves to search for solutions to \eqref{Equation:WeakSolution} along the text.

The following estimate follows from Hölder's Inequality and the Sobolev embeddings $\mathcal{X},H^1\hookrightarrow L^6$; $H^1\hookrightarrow L^{3/2}$.
\begin{lem}\label{Lemma:FourTerms}
We have
\[
	\int\abs{
		K_\epsilon\phi_{\epsilon,u_1 u_2}u_3u_4
	}
	\lesssim
	\norm{u_1}_{H^1}
	\norm{u_2}_{H^1}
	\norm{u_3}_{H^1}
	\norm{u_4}_{H^1}
\]
for every $u_1, \ldots, u_4 \in H^1$ and $\epsilon \in \ooi{0,1}$.
\end{lem}

The regularity of $E_\epsilon$ follows as a corollary.
\begin{lem}\label{Lemma:EnergyClassC2}
The functional $E_\epsilon$ is of class $C^2$. Moreover,
\[
	\Dif_u E_\epsilon \parens{w_1}
	=
	\angles{u \mid w_1}_{H^1_{V_\epsilon}}
	+
	\epsilon^3 \int \parens{
		K_\epsilon \phi_{\epsilon,u^2} u w
	}
	-
	\int\parens{u \abs{u}^{p-1} w_1}
\]
and
\begin{multline*}
	\Dif_u^2 E_\epsilon \parens{w_1,w_2}
	=
	\angles{w_1 \mid w_2}_{H^1_{V_\epsilon}}
	+
	\epsilon^3 \int \brackets*{
			K_\epsilon \parens{
				\phi_{\epsilon, u^2} w_1 w_2
				+
				2 \phi_{\epsilon, u w_1} u w_2
	}}
	+
	\\
	-
	p \int \parens*{\abs{u}^{p-1} w_1 w_2}
\end{multline*}
for every $u, w_1, w_2 \in H^1$.
\end{lem}

Due to the previous lemma, we can characterize solutions to \eqref{Equation:WeakSolution} as critical points of $E_\epsilon$. To finish, we list some elementary inequalities that will be often used throughout the text.
\begin{lem}\label{Lemma:Elementary}
We have
\[
	\abs*{
		\parens{a + b}^q - a^q - q a^{q-1} b
	}
	\lesssim
	\begin{cases}
		\abs{b}^q			&\text{if}~q \leq 2;
		\\
		\abs{b}^2+\abs{b}^q &\text{if}~q > 2;
	\end{cases}
\]
\begin{multline*}
\abs*{
	\parens{a + b_1}^q
	-
	\parens{a + b_2}^q
	-
	q a^{q-1} \parens{b_1 - b_2}
}
\lesssim
\\
\lesssim
\begin{cases}
	\abs{b_1 - b_2}
	\parens{\abs{b_1}^{q-1} + \abs{b_2}^{q-1}}
	&\text{if}~q \leq 2;
	\\
	\abs{b_1 - b_2}
	\parens{
		\abs{b_1}^{q-1}
		+
		\abs{b_2}^{q-1}
		+
		\abs{b_1}
		+
		\abs{b_2}
	}
	&\text{if}~q > 2;\end{cases}
\end{multline*}
and
\[
	\abs*{
		\parens{a + b}^{q-1} - a^{q-1}
	}
	\lesssim
	\begin{cases}
		\abs{b}^{q-1}			&\text{if}~q \leq 2;
		\\
		\abs{b} + \abs{b}^{q-1} &\text{if}~q > 2
	\end{cases}
\]
for every $a \in \cci{-1, 1}$ and $b, b_1, b_2 \in \real$.
\end{lem}

\subsection{Lyapunov--Schmidt reduction}\label{LyapunovSchmidt}
Let us provide an overview of how the \emph{Lyapunov--Schmidt procedure} lets us rewrite the problem
\begin{equation}\label{Equation:CriticalPoint}
	\nabla E_\epsilon\parens{u}=0;
	\quad
	u\in H^1
\end{equation}
as a critical point equation for a functional on a finite-dimensional manifold.

Suppose that $\mathcal{Z}_\epsilon$ is a \emph{manifold of pseudo-critical points} of $E_\epsilon$, i.e., $\mathcal{Z}_\epsilon$ is a finite-dimensional submanifold of $H^1$ and we can conveniently bound $\norm{\nabla E_\epsilon \parens{z}}_{H^1}$ for $z\in\mathcal{Z}_\epsilon$. In particular, we can associate any $z\in\mathcal{Z}_\epsilon$ to an orthogonal decomposition of $H^1$,
\[
	H^1
	=
	\Tan_z \mathcal{Z}_\epsilon
	\oplus
	\parens{\Tan_z \mathcal{Z}_\epsilon}^\perp,
\]
where $\Tan_z \mathcal{Z}_\epsilon$ denotes the tangent space to $\mathcal{Z}_\epsilon$ at $z$. We then rewrite the critical point equation \eqref{Equation:CriticalPoint} as the system of equations
\begin{numcases}
\\
\Pi_{\epsilon,z} \parens{\nabla E_\epsilon \parens{u}}
=
0;
\label{Equation:Auxiliary}
\\
\parens{\Id - \Pi_{\epsilon,z}} \parens{\nabla E_\epsilon \parens{u}}
=
0;
\label{Equation:Bifurcation}
\\
u\in H^1,
\nonumber
\end{numcases}
where
$
	\Pi_{\epsilon,z}
	\colon
	H^1 \to \parens{\Tan_z \mathcal{Z}_\epsilon}^\perp
$
is an orthogonal projection and we respectively name \eqref{Equation:Auxiliary}, \eqref{Equation:Bifurcation} the \emph{auxiliary} and \emph{bifurcation} equations. The next step consists in using the Implicit Function Theorem to prove that if $\epsilon$ is sufficiently small, then we can associate each $z\in\mathcal{Z}_\epsilon$ to a
$
	w_{\epsilon, z} \in \parens{\Tan_z \mathcal{Z}_\epsilon}^\perp
$
such that $\parens{z + w_{\epsilon, z}}$ solves the auxiliary equation. To conclude, it suffices to show that the newly obtained manifold,
\[
\tilde{\mathcal{Z}}_\epsilon
:=
\set{z + w_{\epsilon, z}: z\in\mathcal{Z}_\epsilon}
\subset
H^1,
\]
is a natural constraint of $E_\epsilon$, i.e., a critical point of $E_\epsilon|_{\tilde{\mathcal{Z}}_\epsilon}$ is a critical point of $E_\epsilon$.

\subsection{A related problem}\label{RelatedProblem}

Note that Problem \eqref{Equation:SimpleProblem} is closely related to
\begin{equation}\label{P_lambda}\tag{$P_\lambda$}
	-\Delta u + \lambda^2 u = u\abs{u}^{p-1}~\text{in}~\real^3
\end{equation}
for any $\lambda\in\ooi{0,\infty}$. Indeed, if $u\colon\real^3\to\real$ is a solution to \eqref{Equation:SimpleProblem}, then a straightforward computation shows that the function $x\mapsto\lambda^{2/\parens{p-1}}u\parens*{\lambda x}$ solves \eqref{P_lambda}. We finish by remarking that it suffices to search for critical points of the functional $\overline{I}_\lambda \in C^2\parens{H^1}$ given by
\[
\overline{I}_\lambda\parens{u}
=
\frac{1}{2} \norm{u}_{H^1_\lambda}^2
-
\frac{1}{p+1} \norm{u}_{L^{p+1}}^{p+1}
\]
to obtain weak solutions to \eqref{P_lambda}.

\section{Multiplicity result}\label{MultiplicityResult}
\subsection{The manifold of pseudo-critical points} \label{PseudoMult}

We begin with a Taylor expansion.
\begin{rmk}\label{Remark:TaylorMult}
In view of \ref{V_1},
\[
	\abs*{
		V \parens{\epsilon x} - V \parens{\epsilon \xi}
	}
	\lesssim
	\epsilon \abs{x - \xi}
	\quad\text{and}\quad
	\abs*{
		V \parens{\epsilon x} - V \parens{\epsilon \xi} - \epsilon \nabla V \parens{\epsilon \xi} \cdot \parens{x - \xi}
	}
	\lesssim
	\epsilon^2 \abs{x - \xi}^2
\]
for every $\parens{\epsilon, \xi} \in \ooi{0, 1} \times \real^3$ and $x \in \real^3$.
\end{rmk}

Considering the discussion in Section \ref{RelatedProblem}, we define
$z_{\epsilon, \xi} \colon \real^3 \to \ooi{0, \infty}$
as
\[
	z_{\epsilon, \xi}\parens{x}
	=
	U_{\lambda_{\epsilon \xi}}\parens{x - \xi},
\]
so that $z_{\epsilon, \xi}$ solves $(P_{\lambda_{\epsilon \xi}})$, i.e.,
\begin{equation}\label{Equation:PDE_z_epsilon,xi}
	-\Delta z_{\epsilon, \xi}
	+
	\lambda_{\epsilon \xi}^2 z_{\epsilon, \xi}
	=
	z_{\epsilon, \xi}^p
	~\text{in}~\real^3,
\end{equation}
where we recall that
$\lambda_{\epsilon \xi} = V \parens{\epsilon \xi}^{1/2}$.
Let us state some estimates involving the functions $z_{\epsilon, \xi}$.
\begin{lem}\label{Lemma:z_eps,xi}
We have
\begin{equation}\label{Equation:z_epsilon,xi:1}
	1 \lesssim
	\norm{z_{\epsilon, \xi}}_{H^1},
	\norm{\partial_i z_{\epsilon, \xi}}_{H^1}
	\lesssim 1;
\end{equation}
\begin{equation}\label{Equation:z_epsilon,xi:2}
	\norm{\dot{z}_{\epsilon,\xi,i} + \partial_i z_{\epsilon,\xi}}_{H^1}
	\lesssim
	\epsilon
\end{equation}
and
\begin{equation}\label{Equation:z_epsilon,xi:3}
	\abs{\angles{
		\dot{z}_{\epsilon, \xi, i} \mid \dot{z}_{\epsilon, \xi, j}
	}_{H^1}}
	\lesssim
	\epsilon
	~\text{if}~
	i \neq j
\end{equation}
for every $\parens{\epsilon, \xi} \in \ooi{0, 1} \times \real^3$ and $i \in \set{1, 2, 3}$, where $\dot{z}_{\epsilon, \xi, i}$ denotes the $i$\textsuperscript{th} partial derivative of $\real^3 \ni \zeta \mapsto z_{\epsilon, \zeta} \in H^1$ evaluated at $\xi$.
\end{lem}
\begin{proof}
\emph{Proof of \eqref{Equation:z_epsilon,xi:1}.}
Clearly,
\[
	\parens{\inf V}^{1/\parens{p-1}}
	U \parens*{
		\lambda_{\epsilon \xi} \parens{x-\xi}
	}
	\leq
	z_{\epsilon, \xi} \parens{x}
	\leq
	\parens{\sup V}^{1/\parens{p-1}}
	U \parens*{
		\lambda_{\epsilon \xi} \parens{x-\xi}
	}.
\]
Therefore,
\[
	\parens{\inf V}^{1/\parens{p-1}}
	\leq
	\lambda_{\epsilon \xi}
	\frac{\norm{z_{\epsilon, \xi}}_{H^1}}{\norm{U}_{H^1}}
	\leq
	\parens{\sup V}^{1/\parens{p-1}},
\]
so it follows from \ref{V_1} and \ref{V_2} that
$1 \lesssim \norm{z_{\epsilon, \xi}} \lesssim 1$.

It follows from the differentiation of \eqref{Equation:PDE_z_epsilon,xi} that
\[
	\norm{\partial_i z_{\epsilon, \xi}}_{H^1}
	=
	p \norm*{
		z_{\epsilon, \xi}^{p - 1} \partial_i z_{\epsilon, \xi}
	}_{L^{p / \parens{p - 1}}}
	=
	p \lambda_{\epsilon \xi}^{p / \parens{p - 1}}
	\norm{U \partial_i U}_{L^{p / \parens{p - 1}}},
\]
thus
\[
	p\parens{\inf V}^{p/\parens{p-1}}
	\leq
	\frac{\norm{\partial_i z_{\epsilon, \xi}}_{H^1}}{
		\norm{U \partial_i U}_{L^{p / \parens{p - 1}}}
	}
	\leq
	p\parens{\sup V}^{p/\parens{p-1}}.
\]
In this situation, the other estimate also follows from \ref{V_1} and \ref{V_2}.

\emph{Proof of \eqref{Equation:z_epsilon,xi:2}.}
Differentiating $\zeta \mapsto z_{\epsilon, \zeta}$, we obtain
\[
\dot{z}_{\epsilon, \xi, i} \parens{x}
=
\epsilon \partial_i V \parens{\epsilon \xi}
\brackets*{
	\frac{1}{p-1}
	V \parens{\epsilon \xi}^{-1}
	z_{\epsilon, \xi} \parens{x}
	+
	\frac{1}{2}
	V \parens{\epsilon \xi}^{-1/2}
	\partial_i z_{\epsilon, \xi} \parens{x}
}
-
\partial_i z_{\epsilon, \xi} \parens{x}.
\]
The conclusion follows from \ref{V_1}, \ref{V_2} and \eqref{Equation:z_epsilon,xi:1}.

\emph{Proof of \eqref{Equation:z_epsilon,xi:3}.}
Due to the previous result,
\[
\abs*{
	\angles{
		\dot{z}_{\epsilon, \xi, i} \mid \dot{z}_{\epsilon, \xi, j}
	}_{H^1}
	-
	\angles{
		\partial_i z_{\epsilon, \xi} \mid \partial_j z_{\epsilon, \xi}
	}_{H^1}
}
\lesssim \epsilon.
\]
To conclude, it suffices to prove that
$
\angles{
	\partial_i z_{\epsilon, \xi} \mid \partial_j z_{\epsilon, \xi}
}_{H^1}
=
0
$.
It is clear that
\[
\angles{
	\partial_i z_{\epsilon, \xi} \mid \partial_j z_{\epsilon, \xi}
}_{H^1}
=
\lambda_{\epsilon \xi}^{\parens{5 - p}/\parens{p - 1}}
\angles{
	\partial_i U \mid \partial_j U
}_{H^1},
\]
so it suffices to show that
$
\angles{
	\partial_i U \mid \partial_j U
}_{H^1}
=
0
$.
As
\[
	-\Delta \parens{\partial_i U} + \partial_i U = p U^{p-1} \partial_i U
	~\text{in}~
	\real^3,
\]
we deduce that
$
\angles{
	\partial_i U \mid \partial_j U
}_{H^1}
=
p
\int \brackets{
	\parens{\partial_i U} \parens{\partial_j U} U^{p-1}
}
$.
The function $U$ is spherically symmetric, while
$x \mapsto \partial_i U \parens{x} \partial_j U \parens{x}$
is odd in the $i$\textsuperscript{th} and $j$\textsuperscript{th} variable, hence the result.
\end{proof}

We proceed to define
$\mathcal{Z}_\epsilon = \set{z_{\epsilon, \xi}: \xi \in \real^3}$,
which is easily seen to be a non-compact submanifold of $H^1$ and we remark that its tangent spaces are given by
\[
	\Tan_{z_{\epsilon, \xi}} \mathcal{Z}_\epsilon = \mathrm{span} \set{
		\dot{z}_{\epsilon, \xi, 1},
		\dot{z}_{\epsilon, \xi, 2},
		\dot{z}_{\epsilon, \xi, 3}
	}
	\subset
	H^1.
\]
The manifold $\mathcal{Z}_\epsilon$ will act as our manifold of pseudo-critical points of $E_\epsilon$ due to the result that follows.
\begin{lem}\label{Lemma:PseudoCriticalPoints}
We have
$
	\norm{\nabla E_\epsilon \parens{z_{\epsilon, \xi}}}_{H^1}
	\lesssim
	\epsilon\abs{\nabla V\parens{\epsilon\xi}}
	+
	\epsilon^2
$
for every $\parens{\epsilon,\xi}\in\ooi{0,1}\times\real^3$.
\end{lem}
\begin{proof}
Clearly,
\[
	\Dif_{z_{\epsilon, \xi}} E_\epsilon \parens{u}
	=
	\Dif_{z_{\epsilon, \xi}} I_\epsilon \parens{u}
	+
	\epsilon^3 \int \parens*{
		K_\epsilon
		\phi_{\epsilon, z_{\epsilon, \xi}^2}
		z_{\epsilon, \xi} u
	}.
\]
Due to Lemmas \ref{Lemma:FourTerms} and \ref{Lemma:z_eps,xi},
$
\epsilon^3
\int \abs{
	K_\epsilon
	\phi_{\epsilon, z_{\epsilon, \xi}^2}
	z_{\epsilon, \xi} u
}
\lesssim
\epsilon^3 \norm{u}_{H^1}
$.
The result then follows from \cite[Lemma 8.8]{AmbrosettiMalchiodi2006}.
\end{proof}

\subsection{Solving the auxiliary equation} \label{AuxMult}

Following the steps of the Lyapunov--Schmidt reduction, our next task consists in solving the auxiliary equation. More precisely, the goal of this section is to prove the following lemma.

\begin{lem}\label{Lemma:AuxMult}
There exists $\epsilon_0 \in \ooi{0, 1}$ and an application of class $C^1$,
\begin{equation}\label{Equation:wMapMultiplicity}
	\ooi{0, \epsilon_0} \times \real^3
	\ni
	\parens{\epsilon, \xi}
	\mapsto
	w_{\epsilon,\xi}\in H^1,
\end{equation}
such that given
$
	\parens{\epsilon, \xi} \in \ooi{0, \epsilon_0} \times \real^3
$,
\[
\Pi_{\epsilon, \xi}\parens{
	\nabla E_\epsilon \parens{z_{\epsilon, \xi} + w_{\epsilon, \xi}}
}
=
0
\quad\text{and}\quad
w_{\epsilon, \xi} \in W_{\epsilon, \xi} := \parens{
	\Tan_{z_{\epsilon,\xi}}\mathcal{Z}_\epsilon
}^\perp.
\]
Moreover,
\[
	\norm{w_{\epsilon, \xi}}_{H^1}
	\lesssim
	\epsilon \abs{\nabla V \parens{\epsilon \xi}}
	+
	\epsilon^2;
	\quad
	\norm{\dot{w}_{\epsilon, \xi, i}}_{H^1}
	\lesssim
	\brackets*{
		\epsilon \abs{\nabla V \parens{\epsilon\xi}}
		+
		\epsilon^2
	}^\mu
\]
for every
$\parens{\epsilon, \xi} \in \ooi{0, \epsilon_0} \times \real^3$
and $i\in\set{1,2,3}$, where $\mu:=\min\parens{1,p-1}$.
\end{lem}

We need several preliminary results to prove Lemma \ref{Lemma:AuxMult}. Let us begin by showing that the second derivative of $E_\epsilon$ is coercive on a certain subspace of $H^1$ if $\epsilon$ is sufficiently small.
\begin{lem}\label{Lemma:Coercive}
There exists $\epsilon_0 \in \ooi{0, 1}$ such that
$
	\Dif^2_{z_{\epsilon, \xi}} E_\epsilon \parens{u, u}
	\gtrsim
	\norm{u}_{H^1}^2
$
for every
$\parens{\epsilon, \xi} \in \ooi{0, \epsilon_0} \times \real^3$
and
$
	u \in \parens{
		\mathrm{span}\set{z_{\epsilon, \xi}}
		\oplus
		\Tan_{z_{\epsilon, \xi}} \mathcal{Z}_\epsilon
	}^\perp
$.
\end{lem}
\begin{proof}
Due to Lemmas \ref{Lemma:FourTerms}, \ref{Lemma:EnergyClassC2} and \ref{Lemma:z_eps,xi},
\[
\abs*{
	\Dif^2_{z_{\epsilon, \xi}} E_\epsilon \parens{u, u}
	-
	\Dif^2_{z_{\epsilon, \xi}} I_\epsilon \parens{u, u}
}
\leq
\epsilon^3
\int \abs*{
	K_\epsilon \parens{
		\phi_{\epsilon, z_{\epsilon, \xi}^2} u^2
		+
		2 \phi_{\epsilon, z_{\epsilon, \xi} u} z_{\epsilon, \xi} u
}}
\lesssim
\epsilon^3 \norm{u}_{H^1}^2.
\]
In view of this estimate, the result follows from \cite[Lemma 8.9]{AmbrosettiMalchiodi2006}.
\end{proof}

Let $R_\epsilon \colon \parens{H^1_{V_\epsilon}}'\to H^1_{V_\epsilon}$ denote the Riesz isomorphism and let $A_{\epsilon, \xi} \colon H^1 \to H^1$, $L_{\epsilon, \xi} \colon W_{\epsilon, \xi} \to W_{\epsilon, \xi}$ be given by
\[
	A_{\epsilon, \xi} \parens{u}
	=
	R_\epsilon \parens*{
		\Dif_{z_{\epsilon, \xi}}^2 E_\epsilon \parens{u, \cdot}
	}
	\quad\text{and}\quad
	L_{\epsilon, \xi} \parens{w}
	=
	\Pi_{\epsilon, \xi} \circ A_{\epsilon, \xi} \parens{w}.
\]
We proceed to a sufficient condition for the invertibility of $L_{\epsilon, \xi}$ with arguments loosely based on the proof of \cite[Lemma 8.10]{AmbrosettiMalchiodi2006} (see also \cite[Lemma 3.2]{IanniVaira2008}).

\begin{lem} \label{Lemma:BoundNormInverseL}
There exists $\epsilon_0 \in \ooi{0, 1}$ such that $L_{\epsilon, \xi}$ is invertible and
$
	\norm{L_{\epsilon, \xi}^{-1}}_{\End\parens{W_{\epsilon, \xi}}} \lesssim 1
$
for every
$\parens{\epsilon, \xi} \in \ooi{0, \epsilon_0} \times \real^3$.
\end{lem}
\begin{proof}[Proof of Lemma \ref{Lemma:BoundNormInverseL}]
\emph{First step.} Let us prove that
$
	\norm{
		z_{\epsilon, \xi} - \Pi_{\epsilon, \xi}\parens{z_{\epsilon, \xi}}
	}_{H^1}
	\lesssim
	\epsilon
$
for every $\parens{\epsilon, \xi} \in \ooi{0, 1} \times \real^3$. Due to Lemma \ref{Lemma:z_eps,xi},
\begin{multline*}
\abs*{
	\angles{
		z_{\epsilon, \xi} \mid \dot{z}_{\epsilon, \xi, i}
	}_{H^1_{\lambda_{\epsilon \xi}}}
}
\leq
\abs*{
	\angles{
		z_{\epsilon, \xi} \mid \dot{z}_{\epsilon, \xi, i} + \partial_i z_{\epsilon, \xi}
	}_{H^1_{\lambda_{\epsilon \xi}}}
}
+
\underbrace{
	\abs*{
		\angles{
			z_{\epsilon, \xi} \mid \partial_i z_{\epsilon, \xi}
		}_{H^1_{\lambda_{\epsilon \xi}}}
	}
}_{=0}
\lesssim
\epsilon.
\end{multline*}

\emph{Second step.} We want to show that
$
	\norm{
		A_{\epsilon, \xi} \parens{z_{\epsilon, \xi}}
		+
		\parens{p-1} z_{\epsilon, \xi}
	}_{H^1}
	\lesssim
	\epsilon
$
for every $\parens{\epsilon, \xi} \in \ooi{0, 1} \times \real^3$. Indeed, in light of Lemma \ref{Lemma:EnergyClassC2} and \eqref{Equation:PDE_z_epsilon,xi},
\begin{multline*}
	\Dif^2_{z_{\epsilon, \xi}} E_\epsilon \parens{z_{\epsilon, \xi}, u}
	+
	\parens{p - 1} \angles{
		z_{\epsilon, \xi} \mid u
	}_{H^1_{\lambda_{\epsilon \xi}}}
	=
	\\
	=
	\int \cbrackets{
		\brackets{V_\epsilon - V \parens{\epsilon \xi}}
		z_{\epsilon, \xi} u
	}
	+
	3\epsilon^3 \int \parens*{
			K_\epsilon
			\phi_{\epsilon, z_{\epsilon, \xi}^2}
			z_{\epsilon, \xi}
			u
	}.
\end{multline*}
On one hand, it follows from \eqref{Equation:ExponentialDecay} and Remark \ref{Remark:TaylorMult} that
\[
\int \abs{
	\brackets{V_\epsilon - V \parens{\epsilon \xi}} z_{\epsilon, \xi} u
}
\lesssim
\epsilon
\int
	\abs{x - \xi}
	z_{\epsilon, \xi} \parens{x}
	\abs{u \parens{x}}
\dif x
\lesssim
\epsilon \norm{u}_{H^1}.
\]
On the other hand, it suffices to argue as in the proof of Lemma \ref{Lemma:Coercive} to prove that
$
\int \abs{
	K_\epsilon \phi_{\epsilon, z_{\epsilon, \xi}^2} z_{\epsilon, \xi} u
}
\lesssim
\epsilon^3 \norm{u}_{H^1}
$.

\emph{Conclusion.}
Clearly,
\begin{multline*}
	L_{\epsilon, \xi} \circ \Pi_{\epsilon, \xi} \parens{z_{\epsilon, \xi}}
	=
	-\parens{p-1} \Pi_{\epsilon, \xi} \parens{z_{\epsilon, \xi}}
	+
	\Pi_{\epsilon, \xi} \parens*{
		A_{\epsilon, \xi} \parens{z_{\epsilon, \xi}}
		+
		\parens{p-1} z_{\epsilon, \xi}
	}
	+
	\\
	+
	\Pi_{\epsilon, \xi} \circ A_{\epsilon, \xi} \parens*{
		\Pi_{\epsilon, \xi} \parens{z_{\epsilon, \xi}}
		-
		z_{\epsilon, \xi}
	}.
\end{multline*}
In view of the first and second steps, 
\[
\abs*{
	L_{\epsilon, \xi} \circ \Pi_{\epsilon, \xi} \parens{z_{\epsilon, \xi}}
	+
	\parens{p-1} \Pi_{\epsilon, \xi} \parens{z_{\epsilon, \xi}}
}
\lesssim
\epsilon.
\]
Therefore,
\[
	\norm{L_{\epsilon, \xi} + \parens{p-1} \Id_{V_{\epsilon, \xi, 1}}}_{
		\mathcal{L} \parens{V_{\epsilon, \xi, 1}, W_{\epsilon, \xi}}
	}
	\lesssim \epsilon
	\quad\text{and}\quad
	\norm{
		L_{\epsilon, \xi} + \Pi_{\epsilon, \xi} \circ A_{\epsilon, \xi}
	}_{
		\mathcal{L} \parens{V_{\epsilon, \xi, 2}, W_{\epsilon, \xi}}
	}
	\lesssim \epsilon,
\]
where
\[
	V_{\epsilon,\xi,1}
	:=
	\mathrm{span} \set{\Pi_{\epsilon, \xi} \parens{z_{\epsilon,\xi}}}
	\quad\text{and}\quad
	V_{\epsilon,\xi,2}
	:=
	\parens*{
		\mathrm{span} \set{z_{\epsilon,\xi}} \oplus \Tan_{z_{\epsilon, \xi}} \mathcal{Z}_\epsilon
	}^\perp.
\]
Finally, the result follows from Lemma \ref{Lemma:Coercive}.
\end{proof}

We still need two lemmas before proceeding to the proof of Lemma \ref{Lemma:AuxMult}.

\begin{lem}\label{Lemma:ApproxDer}
We have
\[
\norm*{
	\Dif_{\parens{z_{\epsilon, \xi} + w}} E_\epsilon
	-
	\Dif_{z_{\epsilon, \xi}} E_\epsilon
}_{H^{-1}}
\lesssim
\norm{w}_{H^1}
+
\norm{w}_{H^1}^p;
\]
\[
\norm*{
	\Dif_{\parens{z_{\epsilon, \xi} + w}} E_\epsilon
	-
	\Dif_{z_{\epsilon, \xi}} E_\epsilon
	-
	\Dif^2_{z_{\epsilon, \xi}} E_\epsilon \parens{w, \cdot}
}_{H^{-1}}
\lesssim
\norm{w}_{H^1}^2
+
\norm{w}_{H^1}^p
\]
and
\[
\norm*{
	\Dif^2_{\parens{z_{\epsilon, \xi} + w}} E_\epsilon \parens{u, \cdot}
	-
	\Dif^2_{z_{\epsilon, \xi}} E_\epsilon \parens{u, \cdot}
}_{H^{-1}}
\lesssim
\parens*{
	\norm{w}_{H^1}
	+
	\norm{w}_{H^1}^2
	+
	\norm{w}^{p-1}_{H^1}
}
\norm{u}_{H^1}
\]
for every $\parens{\epsilon, \xi} \in \ooi{0, 1} \times \real^3$ and $u, w \in H^1$.
\end{lem}
\begin{proof}
The estimates are obtained similarly, so we only prove that the first one holds. In view of Remark \ref{Remark:phi} and Lemma \ref{Lemma:EnergyClassC2},
\begin{multline*}
	\Dif_{\parens{z_{\epsilon, \xi} + w}} E_\epsilon \parens{u}
	-
	\Dif_{z_{\epsilon, \xi}} E_\epsilon \parens{u}
	=
	\angles{w \mid u}_{H^1_{V_\epsilon}}
	+
	\epsilon^3
	\int \parens{
		K_\epsilon
		\phi_{\epsilon, z_{\epsilon, \xi}^2}
		u w
	}
	+
	\\
	+
	2\epsilon^3
	\int \brackets*{
		K_\epsilon
		\phi_{\epsilon, z_{\epsilon, \xi} w}
		u
		\parens{z_{\epsilon, \xi} + w}
	}
	+
	\epsilon^3
	\int \brackets*{
		K_\epsilon
		\phi_{\epsilon, w^2} u \parens{z_{\epsilon, \xi} + w}
	}
	+
	\\
	-
\int \cbrackets*{
	\brackets*{
		\parens{z_{\epsilon, \xi}+w}
		\abs{z_{\epsilon, \xi}+w}^{p-1}
		-
		z_{\epsilon, \xi}^p
	}
	u
}.
\end{multline*}
On one hand, it follows from Lemmas \ref{Lemma:FourTerms} and \ref{Lemma:z_eps,xi} that
\begin{multline*}
\left|
	\epsilon^3
	\int \parens{
		K_\epsilon
		\phi_{\epsilon, z_{\epsilon, \xi}^2}
		u w
	}
	+
	2\epsilon^3
	\int \brackets*{
		K_\epsilon
		\phi_{\epsilon, z_{\epsilon, \xi} w}
		u
		\parens{z_{\epsilon, \xi} + w}
	}
	+
\right.
\\
\left.
	+
	\epsilon^3
	\int \brackets*{
		K_\epsilon
		\phi_{\epsilon, w^2} u \parens{z_{\epsilon, \xi} + w}
	}
\right|
\lesssim
\epsilon^3 \norm{w}_{H^1} \norm{u}_{H^1}.
\end{multline*}
On the other hand, the Sobolev embeddings and Lemma \ref{Lemma:Elementary} imply
\[
\abs*{\int \cbrackets*{
	\brackets*{
		\parens{z_{\epsilon, \xi}+w}
		\abs{z_{\epsilon, \xi}+w}^{p-1}
		-
		z_{\epsilon, \xi}^p
	}
	u
}}
	\lesssim
	\parens*{
		\norm{w}_{H^1} + \norm{w}_{H^1}^p
	}
	\norm{u}_{H^1},
\]
hence the result.
\end{proof}

We proceed to the last preliminary result.

\begin{lem}\label{Lemma:Estimate2ndDer}
It holds that
$
\norm{
	\Dif^2_{z_{\epsilon, \xi}} E_\epsilon \parens{
		\dot{z}_{\epsilon, \xi, i}, \cdot
	}
}_{H^{-1}}
\lesssim
\epsilon
$
for every $\parens{\epsilon, \xi} \in \ooi{0, 1} \times \real^3$ and $i \in \set{1, 2, 3}$.
\end{lem}
\begin{proof}
As
$
	\nabla \overline{I}_{\lambda_{\epsilon \xi}}
	\parens{z_{\epsilon, \xi}}
	=
	0
$
for every $\xi \in \real^3$, we deduce that
$
	\Dif^2_{z_{\epsilon, \xi}} \overline{I}_{\lambda_{\epsilon \xi}}
	\parens{\dot{z}_{\epsilon, \xi, i}, \cdot}
	=
	0
$.
Therefore, it follows from Lemma \ref{Lemma:EnergyClassC2} that
\begin{multline*}
\Dif^2_{z_{\epsilon, \xi}} E_\epsilon \parens{
	\dot{z}_{\epsilon, \xi, i}, u
}
=
\Dif^2_{z_{\epsilon, \xi}} E_\epsilon \parens{
	\dot{z}_{\epsilon, \xi, i}, u
}
-
\Dif^2_{z_{\epsilon, \xi}} \overline{I}_{\lambda_{\epsilon \xi}} \parens{
	\dot{z}_{\epsilon, \xi, i}, u
}
=
\\
=
\int \cbrackets{
	\brackets{V_\epsilon - V \parens{\epsilon \xi}}
	\dot{z}_{\epsilon, \xi, i}
	u
}
+
\epsilon^3
\int \brackets*{
	K_\epsilon \parens{
		\phi_{\epsilon, z_{\epsilon, \xi}^2}
		\dot{z}_{\epsilon, \xi, i} u
		+
		2 \phi_{\epsilon, z_{\epsilon, \xi} \dot{z}_{\epsilon, \xi, i}} z_{\epsilon, \xi} u
}}.
\end{multline*}
On one hand, it follows from \eqref{Equation:ExponentialDecay} and Remark \ref{Remark:TaylorMult} that
\[
\abs*{
	\int \cbrackets{
		\brackets{V_\epsilon - V \parens{\epsilon \xi}}
		\dot{z}_{\epsilon, \xi, i}
		u
	}
}
\lesssim
\epsilon \int
	\abs{x - \xi}
	\dot{z}_{\epsilon, \xi, i} \parens{x}
	u \parens{x}
\dif x
\lesssim
\epsilon \norm{u}_{H^1}.
\]
On the other hand, Lemmas \ref{Lemma:FourTerms} and \ref{Lemma:z_eps,xi}
imply
\[
\epsilon^3
\abs*{
	\int \brackets*{
		K_\epsilon \parens{
			\phi_{\epsilon, z_{\epsilon, \xi}^2}
			\dot{z}_{\epsilon, \xi, i} u
			+
			2 \phi_{\epsilon, z_{\epsilon, \xi}
			\dot{z}_{\epsilon, \xi, i}}
			z_{\epsilon, \xi} u
}}}
\lesssim
\epsilon^3 \norm{u}_{H^1}.
\]
\end{proof}

Let us finally prove Lemma \ref{Lemma:AuxMult} with arguments loosely based on those in the proof of \cite[Proposition 8.7]{AmbrosettiMalchiodi2006}.

\begin{proof}[Proof of Lemma \ref{Lemma:AuxMult}]
\emph{Setup.}
Take $\epsilon_0 \in \ooi{0, 1}$ as in Lemma \ref{Lemma:BoundNormInverseL} and fix $\bar{C} \in \ooi{0, \infty}$. We define
$
	\mathcal{H} \colon \mathcal{O} \to H^1 \times \real^3
$
as the application of class $C^1$ given by
\[
	\mathcal{H}\parens{\epsilon, \xi, w, \alpha}
	=
	\begin{pmatrix}
		\nabla E_\epsilon\parens{z_{\epsilon, \xi}+w}
		-
		\sum_{1\leq i\leq 3}\alpha_i\dot{z}_{\epsilon, \xi, i}
		\\
		\sum_{1 \leq i \leq 3} \angles{w \mid \dot{z}_{\epsilon, \xi, i}}_{H^1} e_i
	\end{pmatrix},
\]
where
\[
	\mathcal{O} := \set*{
		\parens{\epsilon, \xi, w, \alpha}
		\in
		\coi{0, \epsilon_0} \times \real^3 \times H^1 \times \real^3:
		\norm{w}_{H^1} \leq \bar{C} \brackets*{
			\epsilon \abs{\nabla V \parens{\epsilon \xi}}
			+
			\epsilon^2
		}
	}.\footnote{
		The reasoning for such a choice of a domain is explained in \cite[Section 3.3.2]{IanniVaira2008}
	}
\]
and $E_0 := \overline{I}_1$. Note that the auxiliary equation
\[
	\Pi_{\epsilon, \xi} \parens{
		\nabla E_\epsilon\parens{z_{\epsilon,\xi} + w}
	} = 0;
	\quad
	w \in W_{\epsilon, \xi}
\]
is solved if, and only if, $\mathcal{H}\parens{\epsilon,\xi,w,\alpha}=0$ for a certain $\alpha\in\real^3$.

\emph{A preliminary result.} We claim that, up to shrinking $\epsilon_0$, \[
	\Dif_{\parens{w,\alpha}} \mathcal{H}_{\epsilon, \xi}
	\colon
	H^1 \times \real^3
	\to
	H^1 \times \real^3
\]
is invertible and
\[
	\norm*{
		\Dif_{\parens{w,\alpha}} \mathcal{H}_{\epsilon, \xi} \parens{v, \beta}
	}_{H^1 \times \real^3}
	\gtrsim
	\norm{v}_{H^1} + \abs{\beta}
\]
for every
$
	\parens{\epsilon, \xi, w, \alpha} \in \mathcal{O}
$
and $\parens{v, \beta} \in H^1 \times \real^3$. Indeed, it follows from Lemma \ref{Lemma:ApproxDer} that
\begin{multline*}
	\norm*{
		\Dif_{\parens{w,\alpha}} \mathcal{H}_{\epsilon, \xi} \parens{v, \beta}
		-
		\begin{pmatrix}
			A_{\epsilon, \xi} \parens{v}
			-
			\sum_{1\leq i\leq 3}\beta_i\dot{z}_{\epsilon, \xi, i}
			\\
			\sum_{1 \leq i \leq 3} \angles{v \mid \dot{z}_{\epsilon, \xi, i}}_{H^1} e_i
		\end{pmatrix}
	}_{H^1 \times \real^3}
	\lesssim
	\\
	\lesssim
	\parens*{
		\norm{w}_{H^1}
		+
		\norm{w}_{H^1}^2
		+
		\norm{w}^{p-1}_{H^1}
	}
	\norm{v}_{H^1}
	\lesssim
	\parens*{
		\epsilon \abs{\nabla V \parens{\epsilon \xi}}
		+
		\epsilon^2
	}^\mu
	\norm{v}_{H^1}.
\end{multline*}
By considering the linear isomorphism
$
	H^1 \times \real^3
	\to
	W_{\epsilon, \xi}
	\times
	\Tan_{z_{\epsilon, \xi}}\mathcal{Z}_\epsilon
	\times
	\real^3
$
and Lemma \ref{Lemma:Estimate2ndDer}, we obtain
\[
\norm*{
	\Dif_{\parens{w,\alpha}} \mathcal{H}_{\epsilon, \xi} \parens{v, \beta}
	-
	\begin{pmatrix}
		L_{\epsilon, \xi}\parens{v^\perp}
		\\
		-\sum_{1\leq i\leq 3}\beta_i\dot{z}_{\epsilon, \xi, i}
		\\
		\sum_{1 \leq i \leq 3}
		\angles{v^\parallel \mid \dot{z}_{\epsilon, \xi, i}}_{H^1} e_i
	\end{pmatrix}
}_{H^1 \times \real^3}
\lesssim
\epsilon^\mu \norm{v}_{H^1},
\]
where
$
	\parens{v^\perp, v^\parallel}
	\in
	W_{\epsilon, \xi} \times \Tan_{z_{\epsilon, \xi}}\mathcal{Z}_\epsilon
$
and $v = v^\perp + v^\parallel$. The result then follows from  Lemma \ref{Lemma:BoundNormInverseL}.

\emph{The mapping \eqref{Equation:wMapMultiplicity}, its regularity and estimation of $\norm{w_{\epsilon, \xi}}_{H^1}$.}
In light of the preliminary result and the equality
$\mathcal{H} \parens{0, \cdot, 0, 0} \equiv 0$,
we can use the Implicit Function Theorem to fix an application of class $C^1$,
\[
	\coi{0, \epsilon_0} \times \real^3
	\ni
	\parens{\epsilon, \xi}
	\mapsto
	\parens{w_{\epsilon, \xi}, \alpha_{\epsilon, \xi}}
	\in
	H^1 \times \real^3,
\]
such that
\begin{equation}\label{Equation:IFT}
	\mathcal{H} \parens{
		\epsilon, \xi, w_{\epsilon, \xi}, \alpha_{\epsilon, \xi}
	}
	=
	0
\end{equation}
and $\alpha_{0, \xi} = w_{0, \xi} = 0$ for every
$
	\parens{\epsilon, \xi} \in \coi{0, \epsilon_0} \times \real^3
$.
The estimate on $\norm{w_{\epsilon, \xi}}_{H^1}$ follows from the definition of $\mathcal{O}$.

\emph{Estimation of $\norm{\dot{w}_{\epsilon,\xi,i}}_{H^1}$.} 
By differentiating \eqref{Equation:IFT}, we deduce that
\[
	0
	=
	\partial_i\mathcal{H}_{
		\epsilon, w_{\epsilon, \xi}, \alpha_{\epsilon, \xi}
	}
	\parens{\xi}
	+
	\Dif_{\parens{w_{\epsilon, \xi}, \alpha_{\epsilon, \xi}}}
	\mathcal{H}_{\epsilon, \xi} \parens{
		\dot{w}_{\epsilon, \xi, i}, \dot{\alpha}_{\epsilon, \xi, i}
	}.
\]
It follows from the preliminary result that
\begin{multline*}
	\norm{\dot{w}_{\epsilon, \xi, i}}_{H^1}
	\leq
	\norm*{
		\brackets*{
			\Dif_{\parens{w_{\epsilon, \xi}, \alpha_{\epsilon, \xi}}}
			\mathcal{H}_{\epsilon, \xi}
		}^{-1}
		\parens*{
			\partial_i\mathcal{H}_{
				\epsilon, w_{\epsilon, \xi}, \alpha_{\epsilon, \xi}
			}
			\parens{\xi}
		}
	}_{H^1 \times \real^3}
	\lesssim
	\\
	\lesssim
	\norm*{
		\partial_i\mathcal{H}_{
			\epsilon, w_{\epsilon, \xi}, \alpha_{\epsilon, \xi}
		}
		\parens{\xi}
	}_{H^1 \times \real^3}.
\end{multline*}
Clearly,
\begin{multline*}
	\norm*{
		\partial_i\mathcal{H}_{
			\epsilon, w_{\epsilon, \xi}, \alpha_{\epsilon, \xi}
		}
		\parens{\xi}
	}_{H^1 \times \real^3}
	=
	\\
	=
	\norm*{
		\begin{pmatrix}
			R_\epsilon\parens*{
				\Dif^2_{\parens{
					z_{\epsilon, \xi}+w_{\epsilon, \xi}
				}}
				E_\epsilon \parens{\dot{z}_{\epsilon, \xi, i}, \cdot}
			}
			-
			\sum_{1 \leq j \leq 3}
			\alpha_{\epsilon, \xi, j} \ddot{z}_{\epsilon, \xi, i, j}
			\\
			\sum_{1 \leq j \leq 3}
			\angles{
				w_{\epsilon, \xi} \mid \ddot{z}_{\epsilon, \xi, i, j}
			}_{H^1}
			e_j
		\end{pmatrix}
	}_{H^1 \times \real^3}
	\lesssim
	\\
	\lesssim
	\norm*{
		\Dif^2_{\parens{
			z_{\epsilon, \xi}+w_{\epsilon, \xi}
		}}
		E_\epsilon \parens{\dot{z}_{\epsilon, \xi, i}, \cdot}
	}_{H^{-1}}
	+
	\abs{\alpha_{\epsilon, \xi}}
	+
	\norm{w_{\epsilon, \xi}}_{H^1}.
\end{multline*}
Considering Lemmas \ref{Lemma:ApproxDer} and \ref{Lemma:Estimate2ndDer}, we obtain
\begin{multline*}
\norm*{
	\Dif^2_{\parens{z_{\epsilon, \xi}+w_{\epsilon, \xi}}}
	E_\epsilon \parens{\dot{z}_{\epsilon, \xi, i}, \cdot}
}_{H^{-1}}
\leq
\\
\leq
\norm*{
	\Dif^2_{\parens{z_{\epsilon, \xi} + w_{\epsilon, \xi}}}
	E_\epsilon \parens{
		\dot{z}_{\epsilon, \xi, i}, \cdot
	}
	-
	\Dif^2_{z_{\epsilon, \xi}}
	E_\epsilon \parens{
		\dot{z}_{\epsilon, \xi, i}, \cdot
	}
}
+
\norm*{
	\Dif^2_{z_{\epsilon, \xi}}
	E_\epsilon \parens{
		\dot{z}_{\epsilon, \xi, i}, \cdot
	}
}_{H^{-1}}
\lesssim
\\
\lesssim
\epsilon
+
\norm{w_{\epsilon, \xi}}_{H^1}
+
\norm{w_{\epsilon, \xi}}_{H^1}^2
+
\norm{w_{\epsilon, \xi}}_{H^1}^{p-1}
\lesssim
\brackets*{
	\epsilon \abs{\nabla V \parens{\epsilon \xi} + \epsilon^2}
}^\mu.
\end{multline*}
In view of \eqref{Equation:IFT}, Lemmas \ref{Lemma:PseudoCriticalPoints} and \ref{Lemma:ApproxDer}, we conclude that
$
	\abs{\alpha_{\epsilon, \xi}}
	\lesssim
	\epsilon \abs{\nabla V \parens{\epsilon \xi}} + \epsilon^2
$.

\end{proof}

\subsection{The reduced functional}

Consider the context of Lemma \ref{Lemma:AuxMult}. Given $\epsilon \in \ooi{0, \epsilon_0}$, we define the \emph{reduced functional} $\Phi_\epsilon \colon \real^3 \to \real$ as
\[
	\Phi_\epsilon \parens{\xi}
	=
	E_\epsilon \parens{z_{\epsilon, \xi} + w_{\epsilon,\xi}}.
\]
In particular, $\Phi_\epsilon$ is of class $C^1$ as a composition of mappings of class $C^1$. Let us prove that, up to shrinking $\epsilon_0$, critical points of $\Phi_\epsilon$ yield critical points of $E_\epsilon$.

\begin{lem}\label{Lemma:NaturalConstraintMultiplicity}
Up to shrinking $\epsilon_0$, the following implication holds: if $\epsilon \in \ooi{0, \epsilon_0}$ and $\nabla \Phi_\epsilon \parens{\xi} = 0$, then
$\nabla E_\epsilon \parens{z_{\epsilon, \xi} + w_{\epsilon, \xi}} = 0$.
\end{lem}
\begin{proof}
By construction,
\[
	\nabla E_\epsilon \parens{z_{\epsilon, \xi} + w_{\epsilon, \xi}}
	=
	\sum_{1 \leq i \leq 3}
	c_{\epsilon, \xi, i} \dot{z}_{\epsilon, \xi, i}
	\in \Tan_{z_{\epsilon, \xi} + w_{\epsilon, \xi}} \mathcal{Z}_\epsilon.
\]
In particular,
\[
	\partial_i \Phi_\epsilon \parens{\xi}
	=
	\angles{
		\dot{z}_{\epsilon, \xi, i} + \dot{w}_{\epsilon, \xi, i}
		\mid
		\nabla E_\epsilon
		\parens{z_{\epsilon, \xi} + w_{\epsilon, \xi}}
	}_{H^1}
	=
	\sum_{1 \leq j \leq 3}
	c_{\epsilon, \xi, j} \angles{
		\dot{z}_{\epsilon, \xi, i} + \dot{w}_{\epsilon, \xi, i}
		\mid
		\dot{z}_{\epsilon, \xi, j}
	}_{H^1}
\]
for every $i \in \set{1, 2, 3}$.

In view of the previous paragraph,
\[
	M_{\epsilon, \xi}
	\begin{pmatrix}
		c_{\epsilon, \xi, 1} \\ c_{\epsilon, \xi, 2} \\ c_{\epsilon, \xi, 3}
	\end{pmatrix}
	=
	0,
	~\text{where}~
	M_{\epsilon, \xi} := \parens*{\angles{
		\dot{z}_{\epsilon, \xi, i} + \dot{w}_{\epsilon, \xi, i}
		\mid
		\dot{z}_{\epsilon, \xi, j}
	}_{H^1}}_{1 \leq i, j \leq 3}.
\]
It follows from Lemmas \ref{Lemma:z_eps,xi} and \ref{Lemma:AuxMult} that, up to shrinking $\epsilon_0$, $M_{\epsilon, \xi}$ is non-singular for every $\parens{\epsilon, \xi} \in \ooi{0, \epsilon_0} \times \real^3$. In this situation, the only solution for the previous linear system is
$
	c_{\epsilon, \xi, 1} = c_{\epsilon, \xi, 2} = c_{\epsilon, \xi, 3} = 0
$.
\end{proof}

Let us introduce a decomposition of $\Phi_\epsilon$ inspired by \cite[(32)]{AmbrosettiMalchiodiSecchi2001}.
\begin{lem}\label{Lemma:DecompMult}
We have
\[
	\Phi_\epsilon \parens{\xi}
	=
	C_0 V \parens{\epsilon \xi}^\theta
	+
	\Lambda_\epsilon \parens{\xi}
	+
	\Omega_\epsilon \parens{\xi}
	+
	\Psi_\epsilon \parens{\xi}
\]
for every $\xi\in\real^3$, where
\[
	C_0:=\parens*{\frac{1}{2}-\frac{1}{p+1}}\norm{U}_{L^{p+1}}^{p+1};
	\quad
	\theta:=\frac{p+1}{p-1}-\frac{3}{2};
\]
\[
	\Lambda_\epsilon\parens{\xi}
	:=
	\frac{1}{2}
	\int\cbrackets{
		\brackets{
			V_\epsilon-V_\epsilon\parens{\xi}
		}
		z_{\epsilon,\xi}^2
	}
	+
	\int\cbrackets{
		\brackets{
			V_\epsilon-V_\epsilon\parens{\xi}
		}
		z_{\epsilon,\xi} w_{\epsilon,\xi}
	};
\]
\[
	\Omega_\epsilon \parens{\xi}
	:=
	\frac{\epsilon^3}{4}
	\int\brackets*{
		K_\epsilon
		\phi_{
			\epsilon,
			\parens{z_{\epsilon,\xi}+w_{\epsilon,\xi}}^2
		}
		\parens{z_{\epsilon,\xi}+w_{\epsilon,\xi}}^2
	}
\]
and
\[
\Psi_\epsilon \parens{\xi}
:=
\frac{1}{2} \norm{w_{\epsilon,\xi}}_{H^1_{V_\epsilon}}^2
-
\frac{1}{p + 1}
\int\brackets*{
	\abs{z_{\epsilon,\xi}+w_{\epsilon,\xi}}^{p+1}
	-
	z_{\epsilon,\xi}^{p+1}
	-
	\parens{p+1} z_{\epsilon,\xi}^p w_{\epsilon,\xi}
}.
\]
\end{lem}
\begin{proof}
Clearly,
\begin{multline*}
	\Phi_\epsilon \parens{\xi}
	=
	\frac{1}{2} \norm{
		z_{\epsilon,\xi}
		+
		w_{\epsilon,\xi}
	}_{H^1_{\lambda_{\epsilon \xi}}}^2
	+
	\frac{1}{2} \int \cbrackets*{
		\brackets{V_\epsilon - V \parens{\epsilon \xi}}
		\parens{z_{\epsilon,\xi} + w_{\epsilon,\xi}}^2
	}
	+
	\\
	+
	\frac{\epsilon^3}{4} \int \brackets*{
		K_\epsilon
		\phi_{
			\epsilon,
			\parens{z_{\epsilon,\xi} + w_{\epsilon,\xi}}^2
		}
		\parens{z_{\epsilon,\xi} + w_{\epsilon,\xi}}^2
	}
	-
	\frac{1}{p+1} \norm{
		z_{\epsilon,\xi} + w_{\epsilon,\xi}
	}^{p+1}_{L^{p+1}}.
\end{multline*}
By developing the squares; summing and subtracting
$
	\norm{z_{\epsilon, \xi}}_{L^{p+1}}^{p+1}/\parens{p+1}
$,
we obtain
\begin{multline*}
\Phi_\epsilon \parens{\xi}
=
\frac{1}{2} \norm{z_{\epsilon,\xi}}_{H^1_{\lambda_{\epsilon \xi}}}^2
-
\frac{1}{p+1} \norm{z_{\epsilon, \xi}}_{L^{p+1}}^{p+1}
+
\\
+
\underbrace{
		\frac{1}{2}
		\int\cbrackets{
			\brackets{V_\epsilon-V_\epsilon\parens{\xi}}
			z_{\epsilon,\xi}^2
	}
	+
	\int\cbrackets{
		\brackets{
			V_\epsilon-V_\epsilon\parens{\xi}
		}
		z_{\epsilon,\xi} w_{\epsilon,\xi}
	}
}_{= \Lambda_\epsilon \parens{\xi}}
+
\\
+
\underbrace{
	\frac{\epsilon^3}{4} \int \brackets*{
		K_\epsilon
		\phi_{
			\epsilon,
			\parens{z_{\epsilon,\xi} + w_{\epsilon,\xi}}^2
		}
		\parens{z_{\epsilon,\xi} + w_{\epsilon,\xi}}^2
	}
}_{= \Omega_\epsilon \parens{\xi}}
	+
	\\
	+
	\frac{1}{2}
	\int\cbrackets{
		\brackets{
			V_\epsilon-V_\epsilon\parens{\xi}
		}
		w_{\epsilon,\xi}^2
	}
	+
	\frac{1}{2} \norm{w_{\epsilon,\xi}}_{H^1_{\lambda_{\epsilon \xi}}}^2
	+
	\\
	-
	\frac{1}{p+1} \norm{
		z_{\epsilon,\xi} + w_{\epsilon,\xi}
	}^{p+1}_{L^{p+1}}
	+
	\frac{1}{p+1}
	\norm{z_{\epsilon, \xi}}_{L^{p+1}}^{p+1}
	+
	\angles{
		z_{\epsilon,\xi} \mid w_{\epsilon,\xi}
	}_{H^1_{\lambda_{\epsilon \xi}}}.
\end{multline*}
Due to \eqref{Equation:PDE_z_epsilon,xi},
\[
	\norm{z_{\epsilon,\xi}}_{H^1_{\lambda_{\epsilon \xi}}}^2
	=
	\norm{z_{\epsilon,\xi}}_{L^{p+1}}^{p+1}
	\quad
	\text{and}
	\quad
	\angles{
		z_{\epsilon, \xi}
		\mid
		w_{\epsilon, \xi}
	}_{H^1_{\lambda_{\epsilon \xi}}}
	=
	\int\parens{
		z_{\epsilon, \xi}^p
		w_{\epsilon, \xi}
	}.
\]
The previous equalities imply
\[
	\Phi_\epsilon \parens{\xi}
	=
	\parens*{\frac{1}{2}-\frac{1}{p+1}}
	\norm{z_{\epsilon, \xi}}_{L^{p+1}}^{p+1}
	+
	\Lambda_\epsilon \parens{\xi}
	+
	\Omega_\epsilon \parens{\xi}
	+
	\Psi_\epsilon \parens{\xi}.
\]
In this situation, the result follows from the definition of $z_{\epsilon, \xi}$.
\end{proof}

To finish, we use the previous decomposition to expand $\Phi_\epsilon$ and $\nabla \Phi_\epsilon$ by arguing as in the proof of \cite[Lemma 8.11]{AmbrosettiMalchiodi2006}.

\begin{lem} \label{Lemma:ExpansionMultiplicity}
We have
\[
\abs*{\Phi_\epsilon \parens{\xi} - C_0 V \parens{\epsilon \xi}^\theta}
\lesssim
\epsilon
\quad\text{and}\quad
\abs*{
	\nabla \Phi_\epsilon \parens{\xi}
	-
	\epsilon a \parens{\epsilon \xi} \nabla V \parens{\epsilon \xi}
}
\lesssim
\epsilon^{1+\mu}
\]
for every
$\parens{\epsilon,\xi} \in \ooi{0, \epsilon_0} \times \real^3$,
where
$
a \parens{\epsilon \xi}
:=
\theta C_0 V \parens{\epsilon \xi}^{\theta - 1}
$.
\end{lem}
\begin{proof}
Consider the decomposition in Lemma \ref{Lemma:DecompMult}.

\emph{Expansion of $\Phi_\epsilon$.}
\emph{-Estimation of $\abs{\Lambda_\epsilon}$.}
In light of Remark \ref{Remark:TaylorMult} and the Cauchy--Schwarz Inequality,
\[
\abs{\Lambda_\epsilon \parens{\xi}}
\lesssim
\epsilon \abs{\nabla V \parens{\epsilon \xi}}
\cbrackets*{
	\int \parens*{\abs{x - \xi} z_{\epsilon, \xi}^2}
	+
	\norm{w_{\epsilon, \xi}}_{H^1}^2
	\brackets*{
		\int \parens*{\abs{x - \xi}^2z_{\epsilon,\xi}^2}
	}^{1/2}
}.
\]
It then follows from \ref{V_1}, \eqref{Equation:ExponentialDecay} and Lemma \ref{Lemma:AuxMult} that
$\abs{\Lambda_\epsilon \parens{\xi}} \lesssim \epsilon$.

\emph{-Estimation of $\abs{\Omega_\epsilon}$.}
Lemmas \ref{Lemma:FourTerms}, \ref{Lemma:z_eps,xi} and \ref{Lemma:AuxMult} imply
$\abs{\Omega_\epsilon \parens{\xi}} \lesssim \epsilon^3$.

\emph{-Estimation of $\abs{\Psi_\epsilon}$.}
Due to Lemma \ref{Lemma:AuxMult},
$\norm{w_{\epsilon, \xi}}_{H^1_{V_\epsilon}}^2 \lesssim \epsilon^2$. In view of Lemma \ref{Lemma:Elementary} and the Sobolev embeddings,
\[
\abs*{
	\int\brackets*{
		\abs{z_{\epsilon,\xi}+w_{\epsilon,\xi}}^{p+1}
		-
		z_{\epsilon,\xi}^{p+1}
		-
		\parens{p+1} z_{\epsilon,\xi}^p w_{\epsilon,\xi}
	}
}
\lesssim
\norm{w_{\epsilon, \xi}}_{H^1}^2
+
\norm{w_{\epsilon, \xi}}_{H^1}^{p + 1}.
\]
By considering Lemma \ref{Lemma:AuxMult} once again, we obtain
\[
\abs*{
	\int\brackets*{
		\abs{z_{\epsilon,\xi}+w_{\epsilon,\xi}}^{p+1}
		-
		z_{\epsilon,\xi}^{p+1}
		-
		\parens{p+1} z_{\epsilon,\xi}^p w_{\epsilon,\xi}
	}
}
\lesssim
\epsilon^2.
\]

\emph{Expansion of $\nabla \Phi_\epsilon$.}
Clearly,
\begin{multline*}
\partial_i \Phi_\epsilon \parens{\xi}
=
\Dif_{z_{\epsilon, \xi}} E_\epsilon \parens{
	\dot{z}_{\epsilon, \xi, i}
}
+
\Dif_{z_{\epsilon, \xi}} E_\epsilon \parens{
	\dot{w}_{\epsilon, \xi, i}	
}
+
\\
+
\brackets*{
	\Dif_{z_{\epsilon, \xi} + w_{\epsilon, \xi}} E_\epsilon 
	\parens{\dot{z}_{\epsilon, \xi, i}}
	-
	\Dif_{z_{\epsilon, \xi}} E_\epsilon
	\parens{\dot{z}_{\epsilon, \xi, i}}
	-
	\Dif^2_{z_{\epsilon, \xi}} E_\epsilon
	\parens{w_{\epsilon,\xi}, \dot{z}_{\epsilon, \xi, i}}
}
+
\\
+
\brackets*{
	\Dif_{z_{\epsilon, \xi} + w_{\epsilon, \xi}} E_\epsilon 
	\parens{\dot{w}_{\epsilon, \xi, i}}
	-
	\Dif_{z_{\epsilon, \xi}} E_\epsilon
	\parens{\dot{w}_{\epsilon, \xi, i}}
	-
	\Dif^2_{z_{\epsilon, \xi}} E_\epsilon
	\parens{w_{\epsilon,\xi}, \dot{w}_{\epsilon, \xi, i}}
}
+
\\
+
\Dif^2_{z_{\epsilon, \xi}} E_\epsilon
\parens{w_{\epsilon,\xi}, \dot{z}_{\epsilon, \xi, i}}
+
\Dif^2_{z_{\epsilon, \xi}} E_\epsilon
\parens{w_{\epsilon,\xi}, \dot{w}_{\epsilon, \xi, i}}.
\end{multline*}
Therefore,
\begin{multline}\label{Equation:Aux}
\abs*{
	\partial_i \Phi_\epsilon \parens{\xi}
	-
	\Dif_{z_{\epsilon, \xi}} E_\epsilon \parens{
		\dot{z}_{\epsilon, \xi, i}
	}
}
\leq
\norm*{\Dif_{z_{\epsilon, \xi}} E_\epsilon}_{H^{-1}}
\norm{\dot{w}_{\epsilon, \xi, i}}_{H^1}
+
\\
+
\norm*{\mathcal{R}_{\epsilon, \xi}}_{H^{-1}}
\parens*{
	\norm{\dot{z}_{\epsilon, \xi, i}}_{H^1}
	+
	\norm{\dot{w}_{\epsilon, \xi, i}}_{H^1}
}
+
\\
+
\norm*{
	\Dif^2_{z_{\epsilon, \xi}} E_\epsilon \parens{
		\dot{z}_{\epsilon, \xi, i}, \cdot
	}
}_{H^{-1}}
\norm{w_{\epsilon,\xi}}_{H^1}
+
\abs*{
	\Dif^2_{z_{\epsilon, \xi}} E_\epsilon
	\parens{w_{\epsilon,\xi}, \dot{w}_{\epsilon, \xi, i}}
},
\end{multline}
where
\[
	\mathcal{R}_{\epsilon, \xi}
	:=
	\Dif_{z_{\epsilon, \xi} + w_{\epsilon, \xi}} E_\epsilon 
	-
	\Dif_{z_{\epsilon, \xi}} E_\epsilon
	-
	\Dif^2_{z_{\epsilon, \xi}} E_\epsilon
	\parens{w_{\epsilon,\xi}, \cdot}
	\in
	H^{-1}.
\]

Let us estimate the terms on the right-hand side of \eqref{Equation:Aux}. Due to Lemmas \ref{Lemma:PseudoCriticalPoints} and \ref{Lemma:AuxMult},
\[
\norm*{\Dif_{z_{\epsilon, \xi}} E_\epsilon}_{H^{-1}}
\norm{\dot{w}_{\epsilon, \xi, i}}_{H^1}
\lesssim
\epsilon^{1 + \mu}.
\]
In view of Lemmas \ref{Lemma:z_eps,xi}, \ref{Lemma:AuxMult} and \ref{Lemma:ApproxDer},
\[
\norm*{
	\mathcal{R}_{\epsilon, \xi}
}_{H^{-1}}
\parens*{
	\norm{\dot{z}_{\epsilon, \xi, i}}_{H^1}
	+
	\norm{\dot{w}_{\epsilon, \xi, i}}_{H^1}
}
\lesssim
\epsilon^{1 + \mu}.
\]
Lemmas \ref{Lemma:AuxMult}, \ref{Lemma:Estimate2ndDer} imply
\[
\norm*{
	\Dif^2_{z_{\epsilon, \xi}} E_\epsilon \parens{
		\dot{z}_{\epsilon, \xi, i}, \cdot
	}
}_{H^{-1}}
\norm{w_{\epsilon,\xi}}_{H^1}
\lesssim
\epsilon^2.
\]
Similar arguments show that
\[
\abs*{
	\Dif^2_{z_{\epsilon, \xi}} E_\epsilon
	\parens{w_{\epsilon,\xi}, \dot{w}_{\epsilon, \xi, i}}
}
\lesssim
\epsilon^{1 + \mu}.
\]

To conclude, it suffices to show that
\[
\abs*{
	\Dif_{z_{\epsilon, \xi}} E_\epsilon \parens{
		\dot{z}_{\epsilon, \xi, i}
	}
	-
	a \parens{\epsilon \xi} \partial_i V \parens{\epsilon \xi}
}
\lesssim
\epsilon^2.
\]
Indeed,
\[
	E_\epsilon \parens{z_{\epsilon, \xi}}
	=
	C_0 V \parens{\epsilon \xi}^\theta
	+
	\frac{1}{2} \int \cbrackets*{
		\brackets{V_\epsilon - V \parens{\epsilon \xi}}
		z_{\epsilon, \xi}^2
	}
	+
	\frac{\epsilon^3}{4} \int \parens*{
		K_\epsilon
		\phi_{\epsilon, z_{\epsilon, \xi}^2}
		z_{\epsilon, \xi}^2
	}.
\]
Differentiating, we obtain
\begin{multline}\label{Equation:Aux2}
	\Dif_{z_{\epsilon, \xi}} E_\epsilon \parens{
		\dot{z}_{\epsilon, \xi, i}
	}
	-
	\epsilon a \parens{\epsilon \xi} \partial_i V \parens{\epsilon \xi}
	=
	\\
	=
	\int
		\brackets{
			V_\epsilon \parens{x}
			-
			V \parens{\epsilon \xi}
			-
			\epsilon \nabla V \parens{\epsilon \xi} \cdot \parens{x - \xi}
		}
		z_{\epsilon, \xi} \parens{x}
		\dot{z}_{\epsilon, \xi, i} \parens{x}
	\dif x
	+
	\\
	+
	\epsilon \int
		\nabla V \parens{\epsilon \xi} \cdot \parens{x - \xi}
		z_{\epsilon, \xi} \parens{x}
		\brackets*{
			\dot{z}_{\epsilon, \xi, i} \parens{x}
			+
			\partial_i z_{\epsilon, \xi} \parens{x}
		}
	\dif x
	+
	\\
	-
	\frac{\epsilon}{2}
	\brackets*{
		2 \int
			\nabla V \parens{\epsilon \xi} \cdot \parens{x - \xi}
			z_{\epsilon, \xi} \parens{x}
			\partial_i z_{\epsilon, \xi} \parens{x}
		\dif x
		+
		\partial_i V \parens{\epsilon \xi}
		\int \parens{z_{\epsilon, \xi}^2}
	}
	+
	\\
	+
	\epsilon^3 \int \parens*{
		K_\epsilon
		\phi_{\epsilon, z_{\epsilon, \xi}^2}
		z_{\epsilon, \xi}
		\dot{z}_{\epsilon, \xi, i}
	}.
\end{multline}
Let us estimate the terms on the right-hand side of \eqref{Equation:Aux2}. It follows from Remark \ref{Remark:TaylorMult} and Lemma \ref{Lemma:z_eps,xi} that the first term may be estimated by $\epsilon^2$. Due to \ref{V_1} and Lemma \ref{Lemma:z_eps,xi}, the second term is estimated by $\epsilon^2$. Consider the third term. An integration by parts shows that
\begin{multline*}
2 \int_0^\infty
	\nabla V \parens{\epsilon \xi} \cdot \parens{x - \xi}
	z_{\epsilon, \xi} \parens{x}
	\partial_i z_{\epsilon, \xi} \parens{x}
\dif x_i
+
\partial_i V \parens{\epsilon \xi}
\int_0^\infty
	z_{\epsilon, \xi} \parens{x}^2
\dif x_i
=
\\
=
\int_0^\infty
	\nabla V \parens{\epsilon \xi} \cdot \parens{x - \xi}
	z_{\epsilon, \xi} \parens{x}^2
\dif x_i
=
0
\end{multline*}
because $x_i \mapsto z_{\epsilon, \xi} \parens{x}$ is even and $x_i \mapsto \nabla V \parens{\epsilon \xi} \cdot \parens{x - \xi}$ is odd. Finally, the last term is estimated by $\epsilon^3$ due to Lemmas \ref{Lemma:FourTerms} and \ref{Lemma:z_eps,xi}.
\end{proof}

\subsection{Proof of Theorem \ref{Theorem:MultiplicitySolutions}}

We need two preliminary results for the proof. The first result is that concentration necessarily occurs around critical points of $V$.
\begin{prop}\label{Proposition:Concentration}
Suppose that $x_0 \in \real^3$, $\epsilon_0 \in \ooi{0, 1}$ and
\[
	\set*{
		u_\epsilon \in H^1:
		\nabla E_\epsilon \parens{u_\epsilon} = 0;
		\epsilon \in \ooi{0,\epsilon_0}
	}
\]
are such that
\begin{equation}\label{Equation:ConcentrationOccurs}
	\norm*{
		u_\epsilon - U_{\lambda_{x_0}} \parens{\cdot-x_0}
	}_{H^1}\to 0
	~\text{as}~
	\epsilon \to 0^+.
\end{equation}
We conclude that $\nabla V\parens{x_0}=0$.
\end{prop}

In fact, the previous proposition is analogous to \cite[Theorem 5.1]{IanniVaira2008} and may be proved accordingly. The second preliminary result is \cite[Theorem 1]{AmbrosettiMalchiodiSecchi2001}, which reads as follows.
\begin{thm}\label{Theorem:Temp}
Let $f \in C^2 \parens{\real^3}$, $\mathcal{M}$ be a compact non-degenerate critical manifold of $f$, $\mathcal{N}$ be a neighborhood of $\mathcal{M}$ and $g \in C^1 \parens{\mathcal{N}}$. We conclude that if $\norm{f-g}_{C^1}$ is sufficiently small, then $g$ has at least $\cupl\parens{\mathcal{M}}+1$ critical points in $\mathcal{N}$.
\end{thm}

We proceed to the proof, based on the arguments in \cite[Proof of Theorem 8.5]{AmbrosettiMalchiodi2006}.

\begin{proof}[Proof of Theorem \ref{Theorem:MultiplicitySolutions}]
\emph{Multiplicity of solutions.} It suffices to apply Theorem \ref{Theorem:Temp}. Indeed, let $f = C_0 V_\epsilon^\theta$; $\mathcal{M} = M$ and $\mathcal{N}$ be a bounded neighborhood of $\mathcal{M}$ in $\real^3$. It follows from Lemma \ref{Lemma:ExpansionMultiplicity} that $\norm{f - \Phi_\epsilon}_{C^1} \to 0$ as $\epsilon \to 0$, hence the result.

\emph{Concentration around points of $M$.}
Up to shrinking $\mathcal{N}$, we can suppose that if $x \in \overline{\mathcal{N}}$ and $\nabla V\parens{x}=0$, then $x \in M$. Due to the previous result, we can fix $\epsilon_0 \in \ooi{0, 1}$ such that given $\epsilon \in \ooi{0, \epsilon_0}$,
$u_\epsilon:=z_{\epsilon, \xi_\epsilon}+w_{\epsilon_n, \xi_\epsilon}$
satisfies $\nabla E_{\epsilon} \parens{u_\epsilon} = 0$. Suppose that $x_0 \in \overline{\mathcal{N}}$ is an accumulation point of $\set{\xi_\epsilon}_{\epsilon \in \ooi{0, \epsilon_0}}$. We deduce from the estimate on Lemma \ref{Lemma:AuxMult} that \eqref{Equation:ConcentrationOccurs} holds up to subsequence. It follows from Proposition \ref{Proposition:Concentration} that $x_0 \in M$, hence the result.
\end{proof}

\section{Concentration results}\label{ConcentrationResults}

Our present goal is to prove Theorems \ref{Theorem:SolutionsNonDegenerate} and  \ref{Theorem:SolutionsDegenerate}, so let us suppose throughout this section that $x_0$ is a critical point of $V$. To simplify the notation, we suppose further that $V\parens{x_0}=1$ and $x_0=0$.

\subsection{Lyapunov--Schmidt reduction}\label{Section:LyapunovConcentration}
As in Section \ref{PseudoMult}, we begin with a Taylor expansion.
\begin{rmk} \label{Remark:Taylor}
Due to \ref{V_1},
\[
	\abs*{
		V_\epsilon \parens{x}
		-
		1
		-
		\frac{\epsilon^2}{2} \Dif^2_0 V \parens{x, x}
	}
	\lesssim
	\epsilon^3 \abs{x}^3
\]
if the hypotheses of Theorem \ref{Theorem:SolutionsNonDegenerate} are satisfied; due to \ref{V_3},
\[
	\abs*{
		V_\epsilon \parens{x}
		-
		1
		-
		\frac{\epsilon^n}{n!} \Dif^n_0 V \parens{x, \ldots, x}
	}
	\lesssim
	\epsilon^{n+1} \abs{x}^{n+1}
\]
if the hypotheses of Theorem \ref{Theorem:SolutionsDegenerate} are satisfied and $n < \infty$, where the estimates hold for every $\parens{\epsilon, x} \in \ooi{0, 1} \times \real^3$.
\end{rmk}

The following estimate is analogous to \cite[Lemma 3.1]{IanniVaira2008} and may be proved accordingly.
\begin{lem}\label{Lemma:EstimateForK}
Suppose that $s$ is a positive integer and $\Dif^i_0K=0$ for every $i\in\set{0,\ldots,s-1}$. Given $\eta \in \ooi{0, \infty}$, we conclude that
\[
	\int\abs{
		K_\epsilon \phi_{\epsilon, u z_\xi} w z_\xi
	}
	\lesssim
	\epsilon^{2s}\norm{u}_{H^1}\norm{w}_{H^1}
\]
for every $u,w\in H^1$ and $\xi\in B_\eta$.
\end{lem}

As Problem \eqref{Equation:SimpleProblem} is invariant by translation, we deduce that
\[
	\mathcal{Z}
	:=
	\set*{
		z_\xi:=U\parens{\cdot-\xi}: \xi\in\real^3
	}
\]
is a non-compact critical manifold of $\overline{I}_1$. Furthermore, \cite[Lemma 4.1]{AmbrosettiMalchiodi2006} shows that $\mathcal{Z}$ is non-degenerate. The next result shows that we can bound the derivative of $E_\epsilon$ at points in $\mathcal{Z}$ similarly as in Lemma \ref{Lemma:PseudoCriticalPoints}.
\begin{lem}\label{Lemma:PseudoCriticalConcentration}
Given $\eta\in\ooi{0,\infty}$, there exists $\epsilon_\eta\in\ooi{0,1}$ such that
\[
	\norm{\nabla E_\epsilon\parens{z_\xi}}_{H^1}
	\lesssim
	\begin{cases}
	\epsilon^2
	&\text{if the hypotheses of Theorem \ref{Theorem:SolutionsNonDegenerate} are satisfied};
	\\
	\epsilon^\gamma
	&\text{if the hypotheses of Theorem \ref{Theorem:SolutionsDegenerate} are satisfied}
	\end{cases}
\]
for every
$
	\parens{\epsilon,\xi}
	\in
	\ooi{0,\epsilon_\eta}
	\times
	B_\eta
$.
\end{lem}

The previous lemma is proved very similarly as \cite[Propositions 3.1, 3.2]{IanniVaira2008}, so we omit its proof. In view of Lemma \ref{Lemma:PseudoCriticalConcentration}, we can prove the result that follows by arguing as in Section \ref{AuxMult}.

\begin{lem} \label{Lemma:SolutionAuxiliaryConcentration}
Given $\eta\in\ooi{0,\infty}$, there exist $\epsilon_\eta\in\ooi{0,1}$ and an application of class $C^1$,
\[
	\ooi{0, \epsilon_\eta} \times B_\eta
	\ni
	\parens{\epsilon, \xi} \mapsto w_{\eta, \epsilon,\xi}
	\in
	H^1,
\]
such that given
$\parens{\epsilon, \xi} \in \ooi{0, \epsilon_\eta} \times B_\eta$,
\[
\Pi_\xi \parens{
	\nabla E_\epsilon \parens{z_\xi + w_{\eta, \epsilon, \xi}}
}
=
0
\quad\text{and}\quad
w_{\eta, \epsilon, \xi} \in W_\xi := \parens{
	\Tan_{z_\xi} \mathcal{Z}
}^\perp.
\]
Furthermore,
\[
	\norm{w_{\eta,\epsilon,\xi}}_{H^1}
	\lesssim
	\epsilon^2;
	\quad \quad \quad \quad
	\norm{\dot{w}_{\eta,\epsilon,\xi,i}}_{H^1}
	\lesssim
	\epsilon^{2\mu}
\]
if the hypotheses of Theorem \ref{Theorem:SolutionsNonDegenerate} are satisfied and
\[
	\norm{w_{\eta,\epsilon,\xi}}_{H^1}
	\lesssim
	\epsilon^\gamma;
	\quad \quad \quad \quad
	\norm{\dot{w}_{\eta,\epsilon,\xi,i}}_{H^1}
	\lesssim
	\epsilon^{\gamma \mu}
\]
if the hypotheses of Theorem \ref{Theorem:SolutionsDegenerate} are satisfied, where the estimates hold for every $\parens{\epsilon,\xi}\in\ooi{0,\epsilon_\eta}\times B_\eta$ and $i\in\set{1,2,3}$.
\end{lem}

Fix $\eta\in\ooi{0,\infty}$; let $\epsilon_\eta\in\ooi{0,1}$ be such that the previous lemmas hold and fix $\epsilon\in\ooi{0,\epsilon_\eta}$. In this situation, we define the \emph{reduced functional} $\Phi_{\eta, \epsilon} \colon B_\eta \to \real$ as
\[
	\Phi_{\eta, \epsilon} \parens{\xi}
	=
	E_\epsilon \parens{z_\xi + w_{\eta,\epsilon,\xi}}.
\]
Once again, $\Phi_{\eta, \epsilon}$ is of class $C^1$.

\begin{lem}
Up to shrinking $\epsilon_\eta$, the following implication holds: if
$\epsilon \in \ooi{0, \epsilon_\eta}$
and
$\nabla \Phi_{\eta, \epsilon} \parens{\xi} = 0$,
then
$\nabla E_\epsilon \parens{z_\xi + w_{\eta, \epsilon, \xi}}=0$.
\end{lem}
\begin{proof}
On one hand,
\[
	\partial_i \Phi_{\eta, \epsilon} \parens{\xi}
	=
	\Dif_{\parens{z_\xi + w_{\eta,\epsilon,\xi}}}
	E_\epsilon \parens{
		\dot{z}_{\xi,i} + \dot{w}_{\eta,\epsilon,\xi,i}
	}
\]
for every $i\in\set{1,2,3}$. On the other hand, \cite[Theorem 2.25]{AmbrosettiMalchiodi2006} shows that
\[
	\tilde{\mathcal{Z}}_\eta:=\set{
		z_\xi+w_{\eta,\epsilon,\xi}: \xi \in B_\eta
	}
\]
is a natural constraint of $E_\epsilon$, hence the result. Alternatively, one may argue as in the proof of Lemma \ref{Lemma:NaturalConstraintMultiplicity}.
\end{proof}

To finish the section, we decompose $\Phi_{\eta,\epsilon}$ similarly as in Lemma \ref{Lemma:DecompMult}.
\begin{lem}\label{Lemma:DecompositionOfPhiConcentration}
We have
\[
	\Phi_{\eta,\epsilon}\parens{\xi}
	=
	C_0
	+
	\Lambda_{\eta, \epsilon} \parens{\xi}
	+
	\Omega_{\eta, \epsilon} \parens{\xi}
	+
	\Psi_{\eta, \epsilon} \parens{\xi}
\]
for every $\xi\in B_\eta$, where $C_0$ was defined in Lemma \ref{Lemma:DecompMult};
\[
	\Lambda_{\eta, \epsilon}\parens{\xi}
	:=
	\frac{1}{2}
	\int \brackets*{\parens{V_\epsilon-1} z_\xi^2}
	+
	\int \brackets*{\parens{V_\epsilon-1} z_\xi w_{\eta,\epsilon,\xi}};
\]
\[
	\Omega_{\eta, \epsilon} \parens{\xi}
	:=
	\frac{\epsilon^3}{4}
	\int\brackets*{
		K_\epsilon
		\phi_{
			\epsilon,
			\parens{z_\xi+w_{\eta,\epsilon,\xi}}^2
		}
		\parens{z_\xi+w_{\eta,\epsilon,\xi}}^2
	}
\]
and
\[
	\Psi_{\eta, \epsilon} \parens{\xi}
	:=
	\frac{1}{2}\norm{w_{\eta,\epsilon,\xi}}_{H^1_{V_\epsilon}}^2
	-
	\int\brackets*{
		\abs{z_\xi+w_{\eta,\epsilon,\xi}}^{p+1}
		-
		z_\xi^{p+1}
		-
		\parens{p+1} z_\xi^p w_{\eta,\epsilon,\xi}
	}.
\]
\end{lem}

\subsection{Proof of Theorem \ref{Theorem:SolutionsNonDegenerate}}

Fix $\eta \in \ooi{0, \infty}$ such that $0$ is the only critical point of $V$ in $\overline{B_\eta}$ and let $\epsilon_\eta \in \ooi{0, 1}$ be such that the lemmas in the previous section hold. Let us proceed to an expansion of $\nabla \Phi_{\eta, \epsilon}$.

\begin{lem}\label{Lemma:ExpansionNonDegenerate}
There exists $\delta\in\ooi{0,\infty}$ such that
\[
	\abs{
		\nabla \Phi_{\eta, \epsilon} \parens{\xi}
		-
		\epsilon^2 \Gamma_1 \parens{\xi}
	}
	\lesssim
	\epsilon^{2+\delta}
\]
for every
$
	\parens{\epsilon,\xi}\in\ooi{0,\epsilon_\eta}\times B_\eta
$,
where
\[
\Gamma_1 \parens{\xi}
:=
\sum_{1 \leq i \leq 3} \brackets*{
	\partial_i^2 V \parens{0} \xi_i \int x_i U \parens{x} \partial_i U \parens{x} \dif x
}
e_i.
\]
\end{lem}
\begin{proof}
Consider the decomposition in Lemma \ref{Lemma:DecompositionOfPhiConcentration}.

\emph{Expansion of $\partial_i \Lambda_{\eta, \epsilon}$.} It is clear that
\[
	\partial_i\Lambda_{\eta, \epsilon}\parens{\xi}
	=
	\int\brackets{
		\parens{V_\epsilon-1} z_\xi \dot{z}_{\xi,i}
	}
	+
	\int\brackets{
		\parens{
			V_\epsilon-1
		}
		\dot{z}_{\xi,i} w_{\epsilon,\xi}
	}
	+
	\int\brackets{
		\parens{
			V_\epsilon-1
		}
		z_{\xi} \dot{w}_{\epsilon,\xi,i}
	}.
\]

\emph{- Expansion of
$
\int \parens{V_\epsilon-1} z_\xi \dot{z}_{\xi,i}
$.}

Due to \eqref{Equation:ExponentialDecay} and Remark \ref{Remark:Taylor},
\begin{multline*}
	\abs*{
		\int\brackets{
			\parens{V_\epsilon-1} z_\xi \dot{z}_{\xi,i}
		}
		-
		\frac{\epsilon^2}{2}
		\int
			\Dif^2_0V\parens{x,x}
			z_{\xi}\parens{x}
			\dot{z}_{\xi,i}\parens{x}
		\dif x
	}
	\lesssim
	\\
	\lesssim
	\epsilon^3
	\int
		\abs{x}^3 z_\xi\parens{x} \abs{\dot{z}_{\xi,i}\parens{x}}
	\dif x
	\lesssim
	\epsilon^3.
\end{multline*}
It follows from a change of variable that
\[
\int
	\Dif^2_0 V\parens{x,x} z_\xi \parens{x} \dot{z}_{\xi,i} \parens{x}
\dif x
=
\int
	\Dif^2_0 V\parens{x + \xi,x + \xi}
	U \parens{x} \partial_i U \parens{x}
\dif x.
\]
Considering that $U$ is an even function and $\partial_i U$ is an odd function, we arrive at
\[
\int
	\Dif^2_0 V\parens{x,x} z_\xi \parens{x} \dot{z}_{\xi,i} \parens{x}
\dif x
=
2 \partial_i^2 V \parens{0} \xi_i
\int
	x_i U \parens{x} \partial_i U \parens{x}
\dif x.
\]

\emph{- Estimation of
$
\int\brackets{
	\parens{V_\epsilon-1}
	\dot{z}_{\xi,i} w_{\epsilon,\xi}
}
$,
$
\int\brackets{
	\parens{V_\epsilon-1}
	z_{\xi} \dot{w}_{\epsilon,\xi,i}
}
$.}
In view of \eqref{Equation:ExponentialDecay}, \ref{V_1} and Lemma \ref{Lemma:SolutionAuxiliaryConcentration},
\[
	\int\abs{
		\parens{V_\epsilon-1} \dot{z}_{\xi,i} w_{\epsilon,\xi}
	}
	\lesssim
	\epsilon^2
	\norm{w_{\epsilon,\xi}}_{H^1}
	\brackets*{
		\int
			\Dif^2_0V\parens{x,x}^2
			\dot{z}_{\xi,i}\parens{x}^2
		\dif x
	}^{1/2}
	\lesssim
	\epsilon^4
\]
and
$
\int\abs{
	\parens{V_\epsilon-1} z_{\xi} \dot{w}_{\epsilon,\xi,i}
}
\lesssim
\epsilon^{2 \parens{1+\mu}}.
$

\emph{Estimation of $\abs{\partial_i \Omega_{\eta, \epsilon}}$.}
Due to Lemmas \ref{Lemma:FourTerms} and \ref{Lemma:SolutionAuxiliaryConcentration}, we obtain
$\abs{\partial_i \Omega_{\eta, \epsilon} \parens{\xi}} \lesssim \epsilon^3$.

\emph{Estimation of $\abs{\partial_i \Psi_{\eta, \epsilon}}$.}
Clearly,
\begin{multline*}
	\abs{\partial_i\Psi_{\eta, \epsilon}\parens{\xi}}
	\leq
	\norm{w_{\epsilon,\xi}}_{H^1}
	\norm{\dot{w}_{\epsilon,\xi,i}}_{H^1}
	+
	\int \abs*{
		\parens{V_\epsilon-1} w_{\epsilon,\xi} \dot{w}_{\epsilon,\xi,i}
	}
	+
	\\
	+
	\parens{p+1}
	\int\abs*{
		\parens*{
			\parens{z_\xi+w_{\epsilon,\xi}}
			\abs{z_\xi+w_{\epsilon,\xi}}^{p-1}
			-
			z_\xi^p
			-
			p z_\xi^{p-1} w_{\epsilon,\xi}
		}
	\dot{z}_{\xi,i}
	}
	+
	\\
	+
	\parens{p+1}
	\int\abs*{
		\brackets*{
			\parens{z_\xi+w_{\epsilon,\xi}}
			\abs{z_\xi+w_{\epsilon,\xi}}^{p-1}
			-
			z_\xi^p
		}
		\dot{w}_{\epsilon,\xi,i}
	}.
\end{multline*}
As $V$ is bounded, we can use Lemma \ref{Lemma:SolutionAuxiliaryConcentration} to conclude that
\[
\norm{w_{\epsilon,\xi}}_{H^1}
\norm{\dot{w}_{\epsilon,\xi,i}}_{H^1}
+
\int \abs*{
	\parens{V_\epsilon-1} w_{\epsilon,\xi} \dot{w}_{\epsilon,\xi,i}
}
\lesssim
\epsilon^{2\parens{1+\mu}}.
\]
In view of Lemmas \ref{Lemma:Elementary} and \ref{Lemma:SolutionAuxiliaryConcentration},
\[
\int\abs*{
		\parens*{
			\parens{z_\xi+w_{\epsilon,\xi}}
			\abs{z_\xi+w_{\epsilon,\xi}}^{p-1}
			-
			z_\xi^p
			-
			p z_\xi^{p-1} w_{\epsilon,\xi}
		}
	\dot{z}_{\xi,i}
	}
\lesssim
\epsilon^{\min\parens{4,2p}}
\]
and
\[
\int\abs*{
	\brackets*{
		\parens{z_\xi+w_{\epsilon,\xi}}
		\abs{z_\xi+w_{\epsilon,\xi}}^{p-1}
		-
		z_\xi^p
	}
	\dot{w}_{\epsilon,\xi,i}
}
\lesssim
\epsilon^{2 \parens{1+\mu}}.
\]
\end{proof}

Now, we argue similarly as in the proof of \cite[Theorem 2.17]{AmbrosettiMalchiodi2006} to prove Theorem \ref{Theorem:SolutionsNonDegenerate}.

\begin{proof}[Proof of Theorem \ref{Theorem:SolutionsNonDegenerate}]
\emph{Existence of critical point.}
The only root of $\Gamma_1$ in $B_\eta$ is $0$, so
$
	\deg \parens{\Gamma_1, B_\eta, 0}\neq 0
$,
where $\deg$ denotes the topological degree. In view of the continuity of 
$
	\deg \parens{\cdot, B_\eta, 0}
$
and Lemma \ref{Lemma:ExpansionNonDegenerate}, we deduce that, up to shrinking $\epsilon_\eta$,
$\deg \parens{\nabla \Phi_{\eta, \epsilon}, B_\eta, 0} \neq 0$
for every $\epsilon \in \ooi{0, \epsilon_\eta}$. We conclude that given $\epsilon \in \ooi{0, \epsilon_\eta}$, $\Phi_{\eta,\epsilon}$ has a critical point $\xi_\epsilon \in B_\eta$ and
$\nabla E_\epsilon \parens{u_\epsilon} = 0$,
where
$u_\epsilon := z_{\xi_\epsilon} + w_{\eta, \epsilon, \xi_\epsilon}$.

\emph{Concentration around $0$.}
It suffices to argue as in the concentration part in the proof of Theorem
\ref{Theorem:MultiplicitySolutions} by taking $\mathcal{M}=\set{0}$; $\mathcal{N} = B_\eta$; $f = \epsilon^2 \Gamma_1$ and $g = \Phi_{\eta, \epsilon}$.
\end{proof}

\subsection{Proof of Theorem \ref{Theorem:SolutionsDegenerate}}

We begin with an expansion of $\nabla \Phi_{\eta, \epsilon}$.
\begin{lem}\label{Lemma:ExpansionDegenerate}
Fix $\eta \in \ooi{0, \infty}$ and let $\epsilon_\eta$ be furnished by Lemma \ref{Lemma:SolutionAuxiliaryConcentration}. We conclude that there exists $\delta\in\ooi{0,\infty}$ such that
\[
	\abs{
		\nabla \Phi_{\eta, \epsilon} \parens{\xi}
		-
		\epsilon^\gamma \Gamma_2 \parens{\xi}
	}
	\lesssim
	\epsilon^{\gamma+\delta}
\]
for every
$
	\parens{\epsilon,\xi}\in\ooi{0,\epsilon_\eta}\times B_\eta
$,
where
\[
	\Gamma_2\parens{\xi}
	:=
	\begin{cases}
		f\parens{\xi}
		&\text{if}~\gamma=n<2m+3;
		\\
		g\parens{\xi}
		&\text{if}~n>2m+3=\gamma;
	\end{cases}
\]
and
\[
f \parens{\xi}
:=
\frac{1}{n!}
\sum_{\substack{0\leq\alpha\leq n \\ 1\leq i\leq 3}} \brackets*{
	\partial_i^n V \parens{0}
	\binom{n}{\alpha}
	\xi_i^{n-\alpha}
	\int
		x_i^\alpha U \parens{x} \partial_i U \parens{x}
	\dif x
}
e_i.
\]
\end{lem}
\begin{proof}
The difference with the proof of Lemma \ref{Lemma:ExpansionNonDegenerate} is that the lowest order term of the expansion of $\partial_i\Phi_{\eta,\epsilon}$ will come from the expansion of
$\partial_i \Lambda_{\eta, \epsilon}$
if $n<2m+3$ and from the expansion of 
$\partial_i \Omega_{\eta, \epsilon}$
if $n>2m+3$.

\emph{Expansion of $\partial_i \Lambda_{\eta, \epsilon}$.}
Suppose that $n = \infty$. Due to \ref{V_1},
\[
	\abs{V_\epsilon \parens{x} - 1}
	=
	\epsilon^L \abs*{
		\int_0^1
			t^L \Dif_{t \epsilon x}^L V \parens{x, \ldots, x}
		\dif t
	}
	\lesssim
	\epsilon^L.
\]
for every $x \in \real^3$ and $L \in \nat$. Therefore, Lemma \ref{Lemma:SolutionAuxiliaryConcentration} implies
$
	\abs{\partial_i \Lambda_{\eta, \epsilon} \parens{\xi}}
	\lesssim
	\epsilon^L
$
for every $L \in \nat$ and $\parens{\epsilon, \xi} \in \ooi{0, \epsilon_\eta} \times \real^3$.

Suppose that $n < \infty$. Let us prove that
\[
	\abs*{
		\partial_i \Lambda_{\eta, \epsilon} \parens{\xi}
		-
		f \parens{\xi} \cdot e_i
	}
	\lesssim
	\epsilon^{n + \min\parens{1, \gamma \mu}}.
\]
Due to \eqref{Equation:ExponentialDecay} and Remark \ref{Remark:Taylor},
\begin{multline*}
	\abs*{
		\int\brackets{
			\parens{V_\epsilon-1} z_\xi \dot{z}_{\xi,i}
		}
		-
		\frac{\epsilon^n}{n!}
		\int
			\Dif^n_0V\parens{x, \ldots, x}
			z_{\xi}\parens{x}
			\dot{z}_{\xi,i}\parens{x}
		\dif x
	}
	\lesssim
	\\
	\lesssim
	\epsilon^{n+1}
	\int
		\abs{x}^{n+1} z_\xi\parens{x} \abs{\dot{z}_{\xi,i}\parens{x}}
	\dif x
	\lesssim
	\epsilon^{n+1}.
\end{multline*}
We can argue as in the analogous situation in the proof of Lemma \ref{Lemma:ExpansionNonDegenerate} to deduce that
\[
\int
	\Dif^n_0V\parens{x, \ldots, x}
	z_{\xi}\parens{x}
	\dot{z}_{\xi,i}\parens{x}
\dif x
=
f \parens{\xi} \cdot e_i.
\]
Let us estimate the remaining term. We have
\[
	\int\abs{
		\parens{V_\epsilon-1} \dot{z}_{\xi,i} w_{\epsilon,\xi}
	}
	\lesssim
	\epsilon^n
	\norm{w_{\epsilon,\xi}}_{H^1}
	\brackets*{
		\int
			\Dif^n_0 V\parens{x, \ldots, x}^2
			\dot{z}_{\xi,i}\parens{x}^2
		\dif x
	}^{1/2}
	\lesssim
	\epsilon^{n + \gamma}
\]
and
$
\int\abs{
	\parens{V_\epsilon-1} z_{\xi} \dot{w}_{\epsilon,\xi,i}
}
\lesssim
\epsilon^{n + \gamma \mu}.
$

\emph{Expansion of $\partial_i \Omega_{\eta, \epsilon}$.}
Suppose that $m = \infty$. Due to Lemma \ref{Lemma:EstimateForK}, we can treat this case similarly as the case $n = \infty$ in the expansion of $\partial_i \Lambda_{\eta, \epsilon}$.

Suppose that $m < \infty$. Let us prove that
\[
	\abs*{
		\partial_i \Omega_{\eta, \epsilon} \parens{\xi}
		-
		g \parens{\xi} \cdot e_i
	}
	\lesssim
	\epsilon^{2m+\min\parens{\gamma \mu, 1}}.
\]
Due to Remark \ref{Remark:phi},
\[
	\partial_i \Omega_{\eta, \epsilon} \parens{\xi}
	=
	\epsilon^3
	\int\brackets*{
		K_\epsilon
		\phi_{
			\epsilon,
			\parens{z_\xi+w_{\eta,\epsilon,\xi}}^2
		}
		\parens{z_\xi+w_{\eta,\epsilon,\xi}}
		\parens{\dot{z}_{\xi, i}+\dot{w}_{\eta, \epsilon, \xi, i}}
	}.
\]
Clearly,
\begin{multline} \label{Equation:Temporary}
\int\brackets*{
	K_\epsilon
	\phi_{
		\epsilon,
		\parens{z_\xi+w_{\eta,\epsilon,\xi}}^2
	}
	\parens{z_\xi+w_{\eta,\epsilon,\xi}}
	\parens{\dot{z}_{\xi, i}+\dot{w}_{\eta, \epsilon, \xi, i}}
}
=
\int\parens{
	K_\epsilon \phi_{\epsilon, z_\xi^2} z_\xi \dot{z}_{\xi, i}
}
+
\\
+ \cbrackets*{
	\int\brackets*{
		K_\epsilon
		\phi_{\epsilon, \parens{z_\xi+w_{\eta,\epsilon,\xi}}^2}
		\parens{z_\xi+w_{\eta,\epsilon,\xi}}
		\parens{\dot{z}_{\xi, i}+\dot{w}_{\eta, \epsilon, \xi, i}}
	}
	-
	\int\parens{
		K_\epsilon \phi_{\epsilon, z_\xi^2} z_\xi \dot{z}_{\xi, i}
	}
}.
\end{multline}

Let us expand the first term on the right-hand side of \eqref{Equation:Temporary}. By definition,
\[
\int\parens{
	K_\epsilon \phi_{\epsilon, z_\xi^2} z_\xi \dot{z}_{\xi, i}
}
=
\int \int
	K_\epsilon \parens{x} K_\epsilon \parens{y}
	\kappa_\epsilon \parens{x - y}
	z_\xi \parens{y}^2 z_\xi \parens{x} \dot{z}_{\xi, i} \parens{x}
\dif x \dif y.
\]
Similarly as in Lemma \ref{Remark:Taylor}, the Taylor expansion of $K, \kappa$ around $0$ shows that
\begin{multline*}
\left|
	\int\parens{
		K_\epsilon \phi_{\epsilon, z_\xi^2} z_\xi \dot{z}_{\xi, i}
	}
	-
	\frac{\epsilon^{2m}}{\parens{m!}^2}
	\int \int
		\Dif^m_0 K \parens{x, \ldots, x}
		\Dif^m_0 K \parens{y, \ldots, y}
		\times
\right.
\\
\left.
		\times
		z_\xi \parens{y}^2 z_\xi \parens{x} \dot{z}_{\xi, i} \parens{x}
	\dif x \dif y
\right|
\lesssim
\epsilon^{2m+1}.
\end{multline*}
To finish, it suffices to do a change of variable.

An application of Lemmas \ref{Lemma:FourTerms}, \ref{Lemma:SolutionAuxiliaryConcentration} shows that second term on the right-hand side of \eqref{Equation:Temporary} may be estimated up to multiplicative constant by $\epsilon^{2m+\gamma \mu}$.

\emph{Estimation of $\abs{\partial_i \Psi_{\eta, \epsilon}}$.}
Done as in the proof of Lemma \ref{Lemma:ExpansionNonDegenerate}.
\end{proof}

We finish by proving the theorem.

\begin{proof}[Proof of Theorem \ref{Theorem:SolutionsDegenerate}]
Fix $\eta\in\ooi{0,\infty}$ such that $0$ is the only root of $\Gamma_2$ in $B_\eta$ and let $\epsilon_\eta \in \ooi{0,1}$ be such that the lemmas in Section \ref{Section:LyapunovConcentration} hold. At this point, it suffices to follow the general argument used to prove Theorem \ref{Theorem:SolutionsNonDegenerate}.
\end{proof}

\sloppy
\printbibliography
\end{document}